\documentclass[submission,copyright,creativecommons]{eptcs}
\providecommand{\event}{ACT2024} 

\usepackage{iftex}
\usepackage{blindtext}
\usepackage{url}
\usepackage{tabularx}
\usepackage[utf8]{inputenc}
\usepackage[utf8]{inputenc}
\usepackage{amsmath}
\usepackage{amsfonts}
\usepackage{color}
\usepackage{slashbox}
\usepackage{amssymb}
\usepackage{amsthm}
\usepackage{array}
\usepackage{rotating}
\usepackage{physics}
\usepackage{tikz}
\usepackage{braids}
\usepackage{tikz-cd}
\tikzset{%
	symbol/.style={%
		draw=none,
		every to/.append style={%
			edge node={node [sloped, allow upside down, auto=false]{$#1$}}}
	}
}
\usepackage[all]{xy}
\usepackage{cite}
\usepackage{mathdots}
\usepackage{pstricks}
\usepackage{mathrsfs}
\usepackage{setspace}
\usepackage{hyperref}
\theoremstyle{theorem}
\newtheorem{theorem}{Theorem}

\numberwithin{theorem}{section}
\newtheorem{lemma}[theorem]{Lemma}
\newtheorem{proposition}[theorem]{Proposition}
\newtheorem{corollary}[theorem]{Corollary}
\theoremstyle{definition}
\newtheorem{definition}[theorem]{Definition}
\newtheorem{example}[theorem]{Example}
\theoremstyle{remark}
\newtheorem{remark}[theorem]{Remark}
\theoremstyle{notation}
\newtheorem{notation}[theorem]{Notation}
\newcommand{\Gr}{\mathbf{Gray}}
\newcommand{\A}{\mathcal{A}}
\newcommand{\B}{\mathcal{B}}

\newcommand{\Cat}{\mathbf{Cat}}

\ifpdf
  \usepackage{underscore}         
  \usepackage[T1]{fontenc}        
\else
  \usepackage{breakurl}           
\fi

\title{Abstract Kleisli Structures on $2$-categories\footnote{This material is based on Chapter 8 of my PhD thesis \cite{Miranda PhD}. I am grateful for having been supported by the MQRES PhD Scholarship 20192497 while conducting this research, and by EPSRC under grant EP/V002325/2 while writing the paper. I thank Richard Garner for suggesting this as a research topic, Steve Lack for his guidance while I was conducting this research, Nicola Gambino and Mike Johnson for their guidance while I was preparing this paper, and J. S. Lemay for his encouragement.}}
\author{Adrian Miranda
\institute{Department of Mathematics\\University of Manchester\\ Manchester\\ United Kingdom}
\email{\quad adrian.miranda@manchester.ac.uk}
}
\def\titlerunning{Abstract Kleisli Structures on $2$-categories}
\def\authorrunning{A. Miranda}
\begin{document}
\maketitle

\begin{abstract}
Führmann introduced Abstract Kleisli structures to model call-by-value programming languages with side effects, and showed that they correspond to monads satisfying a certain equalising condition on the unit. We first extend this theory to non-strict morphisms of monads, and to incorporate $2$-cells of monads. We then further extend this to a theory of abstract Kleisli structures on $2$-categories, characterising when the original pseudomonad can be recovered by the abstract Kleisli structure on its $2$-category of free-pseudoalgebras.
\end{abstract}

\section{Introduction}\label{Chapter Abstract Kleisli Structures}

\subsection{Context and motivation}\label{Section context and motivation}
Abstract Kleisli structures, also known as thunk-force categories, axiomatise structure that one finds on the Kleisli category of any monad, and have been used to provide direct models of the computational $\lambda$-calculus \cite{Fuhrmann}. Their duals, which axiomatise structure on the coKleisli category of any comonad, have also found applications to runnable monads \cite{Erwig Ren Monadification of functional programs} and the theory of cartesian differential categories \cite{Lemay Cartesian Differential Kleisli Categories}. Variants such as cartesian reverse differential categories build upon the latter of these and are used in modern categorical treatments of gradient-based learning \cite{Categorical Foundations of Gradient Based Learning}. Mathematically, abstract Kleisli structures capture precisely those monads whose unit $\eta: 1_B \Rightarrow T$ is the equaliser of $T\eta$ and $\eta_{T}$. This condition is also equivalent to saying that the Eilenberg Moore adjunction of the monad is of \emph{codescent type} which means that the comparison from $B$ to the category of coalgebras for the comonad induced on $B^T$ is fully-faithful.
\\
\\
\noindent Pseudomonads generalise monads to the two-dimensional setting by allowing conditions such as naturality and monad laws to hold up to isomorphism. They have also received attention in computer science \cite{Pacquet Saville Strong Pseudomonads and Premonoidal Bicategories}, where their $2$-cells allow outputs of computations to be considered before being rewritten or identified in a normal form. However, as yet abstract Kleisli structures remain unexplored in the two-dimensional context. We fill this gap in the literature, and lay the mathematical foundations for future work on the two-dimensional $\lambda$-calculus and differential $\lambda$-calculus \cite{Fiore Gambino Hyland Monoidal Bicategories differential linear logic and analytic functors}.
\subsection{Outline}

Section \ref{Section One dimensional abstract Kleisli} reviews and extends the main results on abstract Kleisli structures in the one-dimensional setting (Theorem 5.3, Lemmas 5.6, 5.28 and 5.29 of \cite{Fuhrmann}). These results exhibit abstract Kleisli structures as objects of a reflective full-subcategory of a category of monads and characterise the monads which correspond to abstract Kleisli structures. We contribute a new, concise definition of abstract Kleisli structures and their categories of `thunkable morphisms', using a lifting condition. We also extend the results of \cite{Fuhrmann} in two directions. The first direction of generalisation is from strict morphisms of monads to the more general morphisms of monads introduced in \cite{Formal Theory of Monads} which are in one-to-one correspondence with extensions between Kleisli categories. The second is to describe the two kinds of $2$-cells between abstract Kleisli structures, namely the monad transformations and the more general $2$-cells considered in \cite{FTM2}.
\\
\\
The remaining sections extend this theory to the $2$-categorical setting. Section \ref{Section Abstract Kleisli 2-categories} defines abstract Kleisli structures on $2$-categories, and thunkability in the two-dimensional setting. Section \ref{Section the isobidescent condition} describes a bicategorical limit condition on the unit which is shown in Theorem \ref{Two dimensional Abstract Kleisli equivalent characterisations} to characterise those pseudomonads that are recoverable from the abstract Kleisli structures on their $2$-categories of free pseudoalgebras. We prove some intermediate results towards this goal in Section \ref{Section some intermediate calculations}. Finally, in Theorem \ref{J a unit for 2 abstract kleisli structures} we exhibit abstract Kleisli structures on $2$-categories as a full, reflective sub-$\mathbf{Gray}$ category of two suitable $\Gr$-categories of pseudomonads.

\subsection{Assumed knowledge, conventions and techniques}

We assume familiarity with pseudomonads and their $2$-categories of pseudoalgebras, fixing notation in Definition \ref{definition pseudomonad} and referring to \cite{Doctrines fully faithful adjoint string} for details. Many of our proofs require chasing large pasting diagrams. These are typically omitted for the sake of brevity but can be found in either Chapter 8 or the appendices of the author's Ph. D. thesis \cite{Miranda PhD}. Some sample calculations are included in Appendix \ref{Appendix} for illustrative purposes. In these proofs we freely use the pasting theorem for $2$-categories \cite{Power 2 cat pasting}. A particular bicategorical limit \cite{Kelly elementary observations on 2-categorical limits} called an isobidescent object will be needed, and its relevant properties will be described explicitly in Section \ref{Section the isobidescent condition}. The results of Section \ref{Section Main Results} are expressed in terms of the semi-strict three-dimensional categories \cite{Gurski Coherence in Three Dimensional Category Theory}, or $\Gr$-categories, that pseudomonads form. One of these structures is the Kleisli version of the $\Gr$-category of pseudomonads defined in \cite{Marmolejo Pseudodistributive 1}, while the other extends this structure in the spirit of \cite{FTM2}. Here $\mathbf{Gray}$ denotes $2$-$\Cat$ equipped with the $\mathbf{Gray}$-tensor product \cite{Gray Formal Category Theory}, and a $\Gr$-category is a category enriched over this base.

\section{Abstract Kleisli structures on categories}\label{Section One dimensional abstract Kleisli}

\noindent We begin with a reformulation of the definition of abstract Kleisli structures in Definition \ref{Definition abstract kleisli structure} and the corresponding monad on the category of thunkable morphisms in Proposition \ref{Proposition adjunction and induced comonad}. We then recall $2$-categories of monads from \cite{FTM2} in Notation \ref{notation 2-categories of monads}, and use this to define morphisms and $2$-cells of abstract Kleisli structures in Definition \ref{def morphisms of abstract kleisli}. In contrast, no $2$-cells of abstract Kleisli structures are defined in \cite{Fuhrmann} while their morphisms of abstract Kleisli structures commute with all structure on the nose and correspond to strict morphisms of monads.

\begin{definition}\label{Definition abstract kleisli structure}
	\hspace{1mm}
	\begin{enumerate}
		\item An \emph{abstract Kleisli structure} on a category $B$ consists of
		
		\begin{itemize}
			\item A comonad $\left(Q, \varepsilon, \delta\right)$ on $B$.
			\item A functor $\theta: B_{0} \rightarrow B^Q$ providing a lifting as in the following diagram, where the unlabelled horizontally depicted functors include the discrete category on the set of objects of $B$, and $U^Q \dashv F^Q$ is the co-Eilenberg-Moore adjunction for $(B, Q, \varepsilon, \delta)$. 
		\end{itemize}
		$$\begin{tikzcd}[column sep = 20, row sep = 15, font=\fontsize{9}{6}]
			B_0
			\arrow[rr]
			\arrow[dddd,"Q_0"']
			&&
			B
			\arrow[rr,"F^Q"]
			&&
			B^Q
			\arrow[dddd,"U^Q"]
			\\
			\\
			\\
			\\
			B_0
			\arrow[rrrruuuu,dashed,"\theta"]
			\arrow[rrrr]
			&&&&
			B
		\end{tikzcd}$$
	\item Given an abstract Kleisli structure on $B$, the associated \emph{category of thunkable morphisms} is given as the factorisation of $\theta$ as displayed below, in which $K$ is fully faithful and $\theta'$ is bijective on objects.
	
	$$\begin{tikzcd}
		B_{0} \arrow[r, "\theta'"] & B_\theta \arrow[r, "K"] & B^Q
	\end{tikzcd}$$
	
	\end{enumerate}
\end{definition}

\begin{example}\label{abstract kleisli structure on kleisli category}
	Let $\left(A, S, \eta, \mu\right)$ be a monad. Then the Kleisli category $A_S$ inherits an abstract Kleisli structure with $\left(Q, \varepsilon, \delta\right)$ the comonad induced by the Kleisli adjunction and $\theta_{X} := F_{S}\eta_{X}$. This captures all examples, and gives the concept its name.
\end{example}

\begin{remark}\label{comparison to fuhrmann}
	  Definition \ref{Definition abstract kleisli structure} part (1) indeed recaptures Definition 2.1 of \cite{Fuhrmann}. The latter consists of a co-pointed endofunctor $\left(B, Q, \varepsilon\right)$, an unnatural transformation $\theta: 1_B \nRightarrow Q$ and various axioms amounting to comonad laws for $(Q, \varepsilon, \theta_{Q})$ and coalgebra laws for $\theta_{X}: X \to QX$. The commutativity of the top triangle in Definition \ref{Definition abstract kleisli structure} part (1) amounts to $\delta = \theta_{Q}$, while the morphisms in the intermediate category $B_\theta$ constructed in Definition \ref{Definition abstract kleisli structure} part (2) are indeed the thunkable morphisms described in Definition 2.3 of \cite{Fuhrmann}; the condition for $f: (X, \theta_{X}) \to (Y, \theta_{Y})$ to be a morphism of coalgebras is precisely naturality of the assignment $X \mapsto \theta_{X}$ in the morphism $f: X \to Y$.
\end{remark}

\begin{proposition}\label{Proposition adjunction and induced comonad}
	The composite functor $F_\theta:=$\begin{tikzcd}
		B_\theta \arrow[r, "K"] & B^Q \arrow[r, "U^Q"] & B
	\end{tikzcd} is faithful and has a right adjoint $U_\theta$, such that the comonad induced on $B$ is $\left(Q, \varepsilon, \delta\right)$.
\end{proposition}

\begin{proof}
	First observe that forgetful functors from categories of coalgebras are faithful, and $K$ is faithful by construction, so the composite $F_\theta$ is also faithful. The right adjoint acts as $U_\theta\left(f: X \rightarrow Y\right) = Qf: QX \rightarrow QY$, with these outputs being morphisms of free coalgebras and hence in $B_\theta$. The unit of the adjunction is given by $\theta$, which is itself in $B_\theta$ by the coassociativity axiom for each coalgebra $\left(X, \theta_{X}\right)$. Naturality for $\theta$ as a unit for the adjunction holds by construction of $B_\theta$, while right triangle identity holds in $B_\theta$ by the right unit law for $\left(Q, \varepsilon, \theta_{Q}\right)$ and the left triangle identity holds in $B$ by the unit law for $\left(X, \theta_{X}\right)$ as a coalgebra. Finally, since $\delta = \theta_Q$, we see that the comonad induced on $B$ is indeed $\left(Q, \varepsilon, \delta\right)$.
\end{proof}

\begin{notation}\label{notation 2-categories of monads}
	For $\kappa \in \{\tau, \lambda\}$, the $2$-category $\mathbf{Monads}_{\kappa}$ has objects given by monads and morphisms $(A, S) \to (B, T)$ given by pairs of functors $F: A \to B$ and $\overline{F}: A_S \to B_T$ commuting with Kleisli left adjoints. These will be referred to as \emph{co-morphisms} of monads. A $2$-cell $(\phi, \overline{\phi}): (F, \overline{F}) \Rightarrow (G, \overline{G})$ in $\mathbf{Monads}_{\tau}$ consists of a pair of natural transformations $\phi: F \Rightarrow G$ and $\overline{\phi}: \overline{F} \Rightarrow \overline{G}$ satisfying a commutativity condition with the left adjoints. Meanwhile, a $2$-cell in $\mathbf{Monads}_{\lambda}$ just consists of natural transformation between the Kleisli categories. The $2$-functor $\mathbf{Monads}: \mathbf{Monads}_{\tau} \rightarrow \mathbf{Monads}_{\lambda}$ is similar to the one described in 2.1 of \cite{FTM2} with Kleisli categories instead of Eilenberg-Moore categories. If a $2$-category is denoted with the subscript $\tau$ (resp. $\lambda$) then its $2$-cells will be called \emph{tight} (resp. \emph{loose}). 
\end{notation}

\begin{definition}\label{def morphisms of abstract kleisli}
	Let $\mathbf{AbsKL}_{0}$ be the class of abstract Kleisli structures and let $\tau: \mathbf{AbsKL}_{0} \to \mathbf{Monads}_{\tau}$ be the class function which sends an abstract Kleisli structure to the monad on its category of thunkable morphisms as per Proposition \ref{Proposition adjunction and induced comonad}. The $2$-category $\mathbf{AbsKL}_{\tau}$ is defined as the image of $\tau$, while the $2$-category $\mathbf{AbsKL}_\lambda$ is defined as the image of the composite \begin{tikzcd}
		\mathbf{AbsKL}_{0} \arrow[r, "\tau"] & \mathbf{Monads}_{\tau} \arrow[r, "\mathbf{Monads}"] & \mathbf{Monads}_{\lambda}
	\end{tikzcd}.
\end{definition}

\noindent This perspective is used to define $\mathbf{Gray}$-categories of abstract Kleisli structures on $2$-categories in Definitions \ref{trikleisli extensions} and \ref{Gray-cat of Abs KL of 2-cats}. Proposition \ref{Proposition simplification of morphisms and 2-cells of abstract Kleisli structures}, to follow, re-expresses the data of Definition \ref{def morphisms of abstract kleisli} in terms of compatibility with the data in an abstract Kleisli structure.

\begin{proposition}\label{Proposition simplification of morphisms and 2-cells of abstract Kleisli structures}
	Let $\left(A, P, \pi\right)$ and $\left(B, Q, \theta\right)$ be abstract Kleisli structures and let $F_\pi: A_\pi \rightarrow A$ and $F_\theta: B_\theta \rightarrow B$ be the left adjoints described in Proposition \ref{Proposition adjunction and induced comonad}. Let $\overline{G}: A \rightarrow B$ be a functor. 
	
	\begin{enumerate}
		\item To give $G: A_\pi \rightarrow B_\theta$ such that $\left(G, \bar{G}\right)$ is a morphism of abstract Kleisli structures is to assert that $\overline{G}$ preserves thunkability. That is, if $f: X \rightarrow Y$ satisfies $\pi$-naturality then $\overline{G}f$ satisfies $\theta$-naturality.
		\item Given $(H, \overline{H}): (A, P, \pi) \to (B, Q, \theta)$ another morphism of abstract Kleisli structures, to give a loose $2$-cell $\phi: (G, \overline{G}) \Rightarrow (H, \overline{H})$ is just to give a natural transformation $\overline{\phi}: \overline{G} \Rightarrow \overline{H}$.
		\item Given $\overline{\phi}$ as in part (2), to give a $\phi$ making $\left(\phi, \bar{\phi}\right)$ into a tight $2$-cell is to assert that the components $\overline{\phi}_{X}$ are thunkable.
		\item There is a commutative square of $2$-functors as depicted below, in which the horizontal maps are $2$-fully faithful and send an abstract Kleisli structure $\left(B, Q, \theta\right)$ to the monad induced on $B_\theta$ from the adjunction described in Proposition \ref{Proposition adjunction and induced comonad}. 
		
		$$\begin{tikzcd}[font=\fontsize{9}{6}]
			\mathbf{AbsKL}_{\tau}
			\arrow[rr]
			\arrow[dd, "\mathbf{AbsKL}"']
			&&\mathbf{Monads}_\tau
			\arrow[dd, "\mathbf{Monads}"]
			\\
			\\
			\mathbf{AbsKL}_{\lambda} \arrow[rr] &&\mathbf{Monads}_\lambda
		\end{tikzcd}$$
	\end{enumerate}
\end{proposition}

\begin{proof}
	Parts (1), (2) and (3) are easy to observe using bijectivity on objects and faithfulness of the left adjoints. Part (4) follows by construction.
\end{proof}

\begin{remark}
	Although it will not be needed for any of our proofs, we note that Proposition \ref{Proposition simplification of morphisms and 2-cells of abstract Kleisli structures} part (4) is a fully faithful $\mathbf{BO}$-enriched functor, in the sense of \cite{Lack Miranda Universal Property of 2-category of monads}.
\end{remark}

\begin{theorem}\label{Conditions characterising monads whose EM adjunctions are of codescent type}
	Let $\left(B, T, \eta, \mu\right)$ be a monad, $\underline{T}$ be the comonad induced on the Kleisli category $B_T$, and $\bar{T}$ be the comonad induced on the Eilenberg-Moore category $B^T$. The following are equivalent.
	
	\begin{enumerate}
		\item $\left(B, T, \eta, \mu\right)$ is in the essential image of $I_\kappa: \mathbf{AbsKL}_{\kappa} \to \mathbf{Monads}_{\kappa}$ for $\kappa \in \{\tau, \lambda\}$.
		\item The natural transformation $\eta$ is the equaliser of $T\eta$ and $\eta_{T}$.
		\item The Kleisli left adjoint $F_T: B \rightarrow B_T$ is both faithful and full on thunkable morphisms. 
		\item The canonical comparison $B \rightarrow {\left(B^T\right)}^{\bar{T}}$ is fully faithful.
		\item The canonical comparison $B \rightarrow {\left(B_T\right)}^{\underline{T}}$ is fully faithful.
	\end{enumerate}
\end{theorem}

\begin{proof}
	$(1) \iff (2)$ is Lemma 5.2.8 of \cite{Fuhrmann}. $(2) \iff (3)$ is Lemma 5.2.7 of \cite{Fuhrmann}. $(4) \iff (5)$ is clear since the image of any $X \in B$ under the canonical comparison $B \rightarrow {\left(B^T\right)}^{\bar{T}}$ is a coalgebra for $\bar{T}$ whose underlying algebra for $T$ is free, and is hence also a coalgebra for $\underline{T}$. $(2) \iff (4)$ is a standard result; see Corollary 7 and Theorem 9 of \cite{Toposes Triples and Theories}.
\end{proof}

\begin{theorem}\label{Left adjoints to inclusion of abstract Kleisli on categories}
	Let $\left(A, S\right)$ be a monad, let $\left(A_S, \bar{S}, F_{S}\eta\right)$ be the abstract Kleisli structure of Example \ref{abstract kleisli structure on kleisli category} and let $\bar{\bar{S}}$ be the monad induced on ${\left(A_S\right)}_{F_{S}\eta}$. Then 
	
	\begin{enumerate}
		\item $F_{S}: A \rightarrow A_S$ factorises through the left adjoint $F_{F_{S}\eta}: {\left(A_S\right)}_{F_{S}\eta} \rightarrow A_S$ via a functor $J$.
		\item $\left(J, 1_{A_{S}}\right): \left(A, S\right) \rightarrow \left({\left(A_{S}\right)}_{F_S \eta}, \bar{\bar{S}}\right)$ is a co-morphism of monads.
		\item The co-morphism of monads $\left(J, 1_{A_{S}}\right)$ has the universal property of a unit exhibiting that there are left adjoints to $I_{\kappa}: \mathbf{AbsKL}_{\kappa} \rightarrow \mathbf{Monads}_\kappa$ for $\kappa \in \{\tau, \lambda\}$.
	\end{enumerate}
\end{theorem}

\begin{proof}
	For part (1), to give $J$ is simply to note that $F_{S}\left(f: X \rightarrow Y\right)$ is always thunkable. Part (2) follows immediately from Part (1). For part (3), let $\left(B, T, \pi\right)$ be an abstract Kleisli category. We first consider the one-dimensional aspect of the universal property for $\left(J, 1_{A_S}\right)$ as a unit exhibiting a left adjoint to $I$. By faithfulness of the left adjoints and fully-faithfulness of $I$, it suffices to give a $H'$ as in the diagram below left, for which in turn it suffices to show that $H$ preserves thunkability. Preservation of thunkability can be seen by commutativity of the diagram below right. Note that $\bar{G}Sf = Gf$ for any morphism $f$ in $A$, and that $Gf$ is thunkable by assumption.
	\\
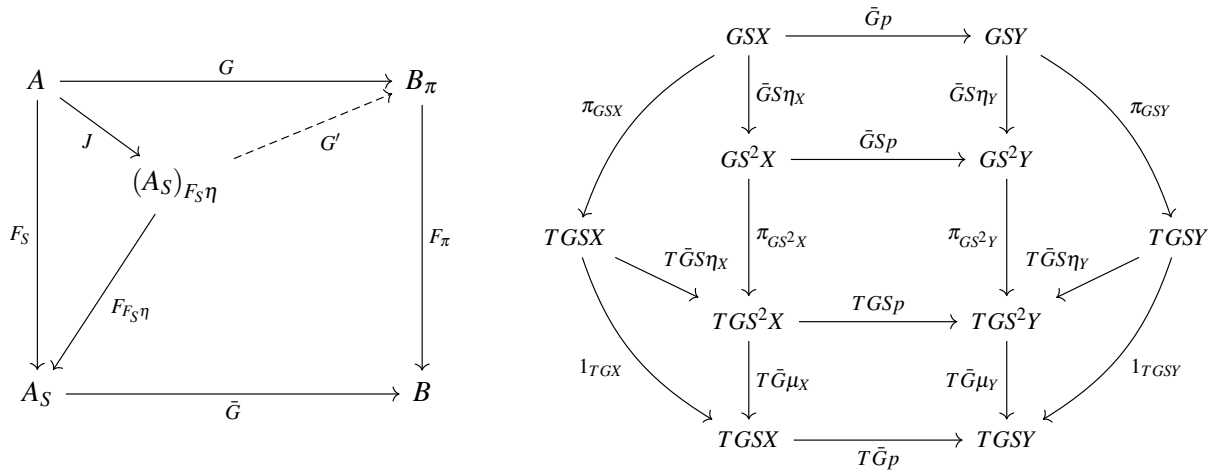
	\begin{tikzcd}[column sep = 20]
		A
		\arrow[rrrr,"G"]
		\arrow[rd,"J"']
		\arrow[dddd,"F_S"']
		&&&&
		B_{\pi}
		\arrow[dddd,"F_{\pi}"]
		\\
		&
		{\left(A_S\right)}_{{F_S}\eta}
		\arrow[rrru,dashed,"G'"']
		\arrow[lddd,"F_{{F_S}\eta}"]
		\\
		\\
		\\
		A_S
		\arrow[rrrr,"\bar{G}"']
		&&&&
		B&{}
	\end{tikzcd}\begin{tikzcd}[column sep = 15, row sep = 15, font=\fontsize{9}{6}]
		&&
		GSX
		\arrow[rrrr,"{\bar{G}}p"]
		\arrow[dd,"{\bar{G}}S\eta_X"]
		\arrow[llddd,bend right=20,"{\pi}_{GSX}"']
		&&&&
		GSY
		\arrow[rrddd,bend left=20,"{\pi}_{GSY}"]
		\arrow[dd,"{\bar{G}}S\eta_Y"']
		\\
		\\
		&&
		G{S^2}X
		\arrow[rrrr,"{\bar{G}}Sp"]
		\arrow[dd, "\pi_{GS^2 X}"]
		&&&&
		G{S^2}Y
		\arrow[dd, "\pi_{GS^2 Y}"']
		\\
		TGSX
		\arrow[rrd, "T\bar{G}S\eta_{X}"]
		\arrow[rrddd, bend right=20, "1_{TGX}"']
		&&&&&&&&
		TGSY
		\arrow[lld, "T\bar{G}S\eta_{Y}"']
		\arrow[llddd,bend left=20, "1_{TGSY}"]
		\\
		&&
		TG{S^2}X
		\arrow[rrrr, "TGSp"]
		\arrow[dd, "T\bar{G}\mu_{X}"]
		&&&&
		TG{S^2}Y
		\arrow[dd, "T\bar{G}\mu_{Y}"']
		\\
		\\
		&&
		TGSX
		\arrow[rrrr, "T\bar{G}p"']
		&&&&
		TGSY
	\end{tikzcd}
	
	\noindent The universal property holds trivially with respect to loose $2$-cells, as they are merely natural transformations between the Kleisli categories and do not need to be factorised. Finally for tight $2$-cells $\left(\phi, \bar{\phi}\right)$, it suffices to see that $\phi$ is natural with respect to thunkable morphisms in $A_S$. But this is true since $\bar{\phi}$ is natural with respect to all morphisms in $A_S$ and $F_{\pi}.\phi = \bar{\phi}.F_S$.
\end{proof}

\section{Categorified thunkability}\label{Section Abstract Kleisli 2-categories}

We now categorify the notion of abstract Kleisli structures and their categories of thunkable morphisms to the context of $2$-categories. As is expected in the process of categorification, thunkability in this context will be a property of $2$-cells but structure on $1$-cells.

\begin{definition}\label{definition pseudomonad}
	A \emph{pseudomonad} on a $2$-category $\A$ consists of a $2$-functor $S: \A \to \A$, two pseudonatural transformations $\eta: 1_\A \Rightarrow S$, $\mu: S^2 \Rightarrow S$, and three invertible modifications $\lambda: \mu.\eta_{S} \Rrightarrow 1_{S}$, $\alpha: \mu.S\mu \Rrightarrow \mu.\mu_{S}$ and $\rho: 1_{S} \Rrightarrow \mu.S\eta$, satisfying the coherences (1)-(5) as listed in Section 8 of \cite{Doctrines fully faithful adjoint string}. A \emph{pseudocomonad} is analogous, but with pseudonatural transformations $\varepsilon: S \Rightarrow 1_\A$ and $\delta: S \Rightarrow S^2$ in place of $\eta$ and $\mu$, respectively.
\end{definition}

\begin{definition}\label{Definition Abstract Kleisli Structure on a 2-category}
	Let $\mathcal{B}$ be a $2$-category. An abstract Kleisli structure $\left(Q, \theta\right)$ on $\mathcal{B}$ consists of
	
	\begin{itemize}
		\item A pseudocomonad $\left(Q, \varepsilon, \delta, \lambda,\alpha, \rho\right)$ on $\mathcal{B}$.
		\item A $2$-functor $\theta: \mathcal{B}_{0} \rightarrow \mathcal{B}^Q$ providing a lifting as in the following diagram, wherein $\B_{0}$ is the set of objects of $\B$ and $U^Q \dashv F^Q$ is the co-Eilenberg-Moore pseudoadjunction. 
		
		$$\begin{tikzcd}[column sep = 15, row sep = 15, font=\fontsize{9}{6}]
			\mathcal{B}_0
			\arrow[rr]
			\arrow[dddd,"Q_0"']
			&&
			\mathcal{B}
			\arrow[rr,"F^Q"]
			&&
			\mathcal{B}^Q
			\arrow[dddd,"U^Q"]
			\\
			\\
			\\
			\\
			\mathcal{B}_0
			\arrow[rrrruuuu,dashed,"\theta"]
			\arrow[rrrr]
			&&&&
			\mathcal{B}
		\end{tikzcd}$$
	\end{itemize}
	The data of $\theta \left(X\right)$ will have its structure map written as $\theta_{X}: X \rightarrow QX$, counitor written as $u_{X}: 1_{X} \Rightarrow \varepsilon_{X}.\theta_{X}$ and coassociator written as $m_{X}: \delta_{X}.\theta_{X} \Rightarrow Q\theta_{X}.\theta_{X}$.
\end{definition}

\noindent Note that although we use the subscript $X$ under $\theta$, the assignment does not extend to a pseudonatural transformation. Similarly, nothing can be said about how  $u$ and $m$ vary with $X$. However, the lifting condition says that $\theta_{QX} = \delta_{X}$, $u_{QX} = \rho_{X}$ and $m_{QX} = \alpha_{X}$.

\begin{example}\label{abstract kleisli structure on free pseudoalgebras}
	Let $\left(\mathcal{A}, S, \eta, \mu, \lambda, \alpha, \rho\right)$ be a pseudomonad. Then the $2$-category of free pseudoalgebras and pseudomorphisms inherits an abstract Kleisli structure. The pseudocomonad is the one induced by the evident pseudoadjunction while the pseudocoalgebra associated to $\left(SX, \mu_{X}\right)$ has structure map given by $\left(S\eta_{X}, \mu_{\eta_{X}}\right)$, counitor given by $\rho_{X}$ and coassociator given by $S\eta_{\eta_{X}}$.
\end{example}

\noindent In Proposition \ref{theta pseudonat 2-category from abstract kleisli structure}, to follow, we give the construction of the $2$-category of `morphisms equipped with thunkings, and thunkable $2$-cells' associated to an abstract Kleisli structure on a $2$-category. As anticipated, thunkability is a property of a $2$-cell but structure on a $1$-cell.

\begin{proposition}\label{theta pseudonat 2-category from abstract kleisli structure}
	(Appendix \ref{proof Proposition theta pseudonat 2-category from abstract kleisli structure}) Let $\mathcal{B}$ be a $2$-category equipped with an abstract Kleisli structure $\left(Q, \varepsilon, \delta, \theta\right)$, and let $\B_\theta$ denote the bijective on objects, $2$-fully faithful factorisation of $\theta: \B_{0} \to \B^Q$, as depicted below left. Then there is a pseudoadjunction as depicted below right in which $F$ is bijective on objects and faithful on $2$-cells. Moreover, the induced pseudocomonad on $\B$ is $(Q, \varepsilon, \delta, \lambda, \alpha, \rho)$.
	
	$$\begin{tikzcd}
		\B_{0} \arrow[r] & \B_\theta \arrow[r] & \B^Q
	\end{tikzcd}\qquad \begin{tikzcd}
	\B \arrow[rr, shift right = 2, "U"'] & \bot& \B_\theta \arrow[ll, shift right = 2, "F"']
\end{tikzcd}$$
\end{proposition}

\begin{definition}
	When $\left(f, \theta_{f}\right)$ is a morphism of $\mathcal{B}_{\theta}$, $\theta_{f}$ will be called a \emph{thunking} of $f$. In this case $f$ may be referred to as a thunked morphism and the pair $\left(f, \theta_{f}\right)$ will be called a morphism equipped with a thunking. The $2$-cells in $\mathcal{B}_\theta$ will be called \emph{thunkable}.
\end{definition}

\begin{proposition}\label{base to theta pseudonat}
	Let $\left(\mathcal{B}, T\right)$ be a pseudomonad and consider the $2$-category $\left({\B}_{T}\right)_{\theta}$ formed by applying the construction of Proposition \ref{theta pseudonat 2-category from abstract kleisli structure} on the abstract Kleisli structure described in Example \ref{abstract kleisli structure on free pseudoalgebras}. Then the left pseudoadjoint $\mathcal{B} \rightarrow {\B}_{T}$ factorises through the left pseudoadjoint $\left({\B}_{T}\right)_{\theta} \rightarrow {\B}_{T}$ via a $2$-functor $J: \mathcal{B} \rightarrow \left({\B}_{T}\right)_{\theta}$. 
\end{proposition}

\begin{proof}
	The morphism $J(f: X \rightarrow Y)$ has thunking given by the $2$-cell $T\eta_{f} \in {\B}_{T}$. That this is well-defined as a $2$-cell of pseudoalgebras follows from pseudonaturality of $\mu$ on $\eta_{f}$. That $T\eta_{f}$ does indeed equip $\left(Tf, \mu_{f}\right)$ with a well-defined thunking follows from the modification coherences for $\rho$ and $\eta_{\eta}$ on $f$. Finally, pseudonaturality of $T\eta$ on $\phi: f \Rightarrow g$ ensures that $J\phi$ is well-defined as a $2$-cell in $\left({\B}_{T}\right)_{\theta}$.
\end{proof}

\section{The isobidescent condition}\label{Section the isobidescent condition}

\noindent We will show in Theorem \ref{J a unit for 2 abstract kleisli structures} that $J$ is a suitable unit exhibiting abstract Kleisli structures on $2$-categories as objects of a full reflective sub-$\mathbf{Gray}$-category of $\mathbf{KLExt}\left(\mathbf{Gray}\right)$. Before doing this we describe in Theorem \ref{Two dimensional Abstract Kleisli equivalent characterisations} certain properties pseudomonads might have which are equivalent to $J$ being a biequivalence. These conditions can be seen as categorifications of those in Theorem \ref{Conditions characterising monads whose EM adjunctions are of codescent type} from the setting of monads to the setting of pseudomonads. In the monads setting, one of these conditions is that $\eta$ is the equaliser of $T\eta$ and $\eta_{T}$. As we will show, in the context of pseudomonads we also have a limit condition characterising those pseudomonads which correspond to abstract Kleisli structures on $2$-categories. However, this equaliser condition is now replaced with the requirement that $\left(1_{T}, \eta, \eta_{\eta}\right)$ exhibit $1_\B$ as an isobidescent object. We recall this notion in the definition below.  

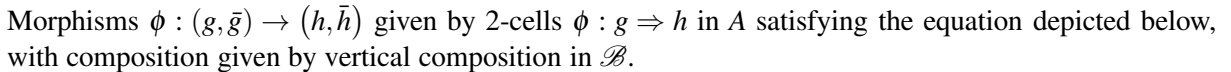
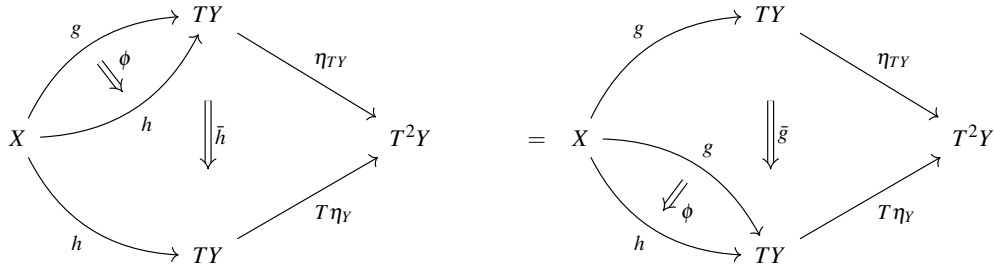
\begin{definition}\label{unit bidescent def}
	Given a pseudomonad $\left(\mathcal{B}, T\right)$ and $X, Y \in \mathcal{B}$, define the category of \emph{descent cones} from $X$ to $Y$ to have objects consisting of data of the form $\left(g, \bar{g}\right)$ where \begin{itemize}
		\item $g: X \rightarrow TY$ is a $1$-cell.
		\item $\bar{g}: \eta_{TY}.g \Rightarrow T\eta_{Y}.g$ is an invertible $2$-cell.
		\item (Unit condition) The following pasting an identity.
		
		\begin{tikzcd}[font=\fontsize{9}{6}, row sep = 12]
			&&TY
			\arrow[rrdd, "\eta_{TY}"]
			\arrow[dddd, Rightarrow, "\bar{g}", shorten = 20]
			\arrow[rrrrdd, "1_{TY}", bend left]
			&&{}\arrow[dd, Rightarrow, "\lambda_{Y}^{-1}", shorten = 8]
			\\
			\\
			X\arrow[rruu, "g"]\arrow[rrdd, "g"']
			&&&&T^2Y
			\arrow[rr, "\mu_{Y}"]
			\arrow[dd, Rightarrow, "\rho_{Y}^{-1}", shorten = 10]
			&&TY
			\\
			\\
			&&TY
			\arrow[rruu, "T\eta_{Y}"']
			\arrow[rrrruu, "1_{TY}"', bend right]
			&&{}
		\end{tikzcd}
		\item (Cocycle condition) The following equation holds.
		\\
		
		\begin{tikzcd}[column sep = 18, font=\fontsize{9}{6}]
			&&TY
			\arrow[rrdd, "\eta_{TY}"]
			\arrow[dddd, Rightarrow, "\bar{g}", shorten = 30]
			\arrow[rr, "T\eta_{Y}"]
			&&T^2 Y
			\arrow[dd, Rightarrow, "T\eta_{\eta_{Y}}", shorten = 10]
			\arrow[rrdd, "T^2 \eta_{Y}"]
			\\
			\\
			X\arrow[rruu, "g"]\arrow[rrdd, "g"']
			&&&&T^2Y
			\arrow[rr, "T\eta_{TY}"]
			\arrow[dd, Rightarrow, "\eta_{\eta_{TY}}", shorten = 10]
			&&T^3Y&=
			\\
			\\
			&&TY
			\arrow[rruu, "T\eta_{Y}"']
			\arrow[rr, "\eta_{TY}"']
			&&T^2Y\arrow[rruu, "\eta_{T^2Y}"']
		\end{tikzcd}\begin{tikzcd}[column sep = 18, font=\fontsize{9}{6}]
			&&TY
			\arrow[dd, Rightarrow, "\bar{g}^{-1}", shorten = 10]
			\arrow[rr, "T\eta_{Y}"]
			&&T^2 Y
			\arrow[dddd, Rightarrow, "\eta_{T\eta_{Y}}", shorten = 10]
			\arrow[rrdd, "T^2 \eta_{Y}"]
			\\
			\\
			X
			\arrow[rruu, "g"]
			\arrow[rrdd, "g"']
			\arrow[rr, "g"]
			&&TY
			\arrow[dd, Rightarrow, "\bar{g}^{-1}", shorten = 10]
			\arrow[rruu, "\eta_{TY}"]
			\arrow[rrdd, "T\eta_{Y}"]
			&&
			&&T^3Y
			\\
			\\
			&&TY
			\arrow[rr, "\eta_{TY}"']
			&&T^2Y\arrow[rruu, "\eta_{T^2Y}"']
		\end{tikzcd}
	\end{itemize}
	Morphisms $\phi: \left(g, \bar{g}\right) \rightarrow \left(h, \bar{h}\right)$ given by $2$-cells $\phi: g \Rightarrow h$ in $A$ satisfying the equation depicted below, with composition given by vertical composition in $\mathcal{B}$.
	
	$$\begin{tikzcd}[font=\fontsize{9}{6}, row sep = 15]
		&&
		TY
		\arrow[rrdd,"{\eta}_{TY}"]
		\arrow[dddd,shorten=25,Rightarrow,"\bar{h}"]
		\\
		\\
		X
		\arrow[rruu,bend left,"g" name=A]
		\arrow[rruu,bend right,"h"' name=B]
		\arrow[rrdd,bend right,"h"']
		&&&&
		{T^2}Y&=
		\\
		\\
		&&
		TY
		\arrow[rruu,"T{\eta}_{Y}"']
		\arrow[from=A,to=B,"\phi",Rightarrow,shorten=8]
	\end{tikzcd}\begin{tikzcd}[font=\fontsize{9}{6}, row sep = 15]
		&&
		TY
		\arrow[rrdd,"{\eta}_{TY}"]
		\arrow[dddd,shorten=25,Rightarrow,"\bar{g}"]
		\\
		\\
		X
		\arrow[rruu,bend left,"g"]
		\arrow[rrdd,bend left,"g" name=A]
		\arrow[rrdd,bend right,"h"' name=B]
		&&&&
		{T^2}Y
		\\
		\\
		&&
		TY
		\arrow[rruu,"T{\eta}_{Y}"']
		\arrow[from=A,to=B,"\phi",Rightarrow,shorten=8]
	\end{tikzcd}$$
	
	\noindent This category will be denoted as $\mathbf{Cone}_{T}\left(X, Y\right)$. We will say that $\left(\mathcal{B}, T\right)$ satisfies \emph{isobidescent} if the canonical functor $\left(\eta \circ -, \eta_\eta \circ - \right): \mathcal{B}\left(X, Y\right) \rightarrow \mathbf{Cone}_{T}\left(X, Y\right)$ which sends $g$ to $\left(\eta_{Y}.g, \eta_{\eta_{Y}}.g\right)$ is an equivalence of categories.
\end{definition}

\noindent Note that $\left(\mathcal{B}, T\right)$ satisfies the isobidescent condition precisely if for every object $Y$ the data $\left(Y, \eta_{Y}, \eta_{\eta_{Y}}\right)$ present $Y$ as a bicategorical version of a descent object. Recall that bilimits have universal properties which hold up to pseudonatural biequivalence. This particular bilimit is dual to the one featuring in the monadicity theorem for pseudomonads \cite{Beck's Theorem For Pseudomonads}.

\section{Some intermediate results}\label{Section some intermediate calculations}

\noindent We wish to show that the $2$-functor $J$ of Proposition \ref{base to theta pseudonat} is a biequivalence if and only if $\left(\mathcal{B}, T\right)$ satisfies isobidescent. Our route towards this will be as follows.

\begin{itemize}
	\item In Proposition \ref{cones to theta pseudonat functor} we will describe functors displayed below.
	$$\underline{J}_{X, Y}: \mathbf{Cone}_{T}\left(X, Y\right) \rightarrow \left({\B}_{T}\right)_{\theta}\big(\left(TX, \mu_{X}\right), \left(TY, \mu_{Y}\right)\big)$$
	\item In Proposition \ref{cones to free theta pseudonat an equivalence} we will prove that these functors are equivalences.
	\item In Proposition \ref{G isomorphic to restriction along eta} we will show that there are natural isomorphisms $\underline{J}_{X, Y}.\left(\eta \circ -, \eta_{\eta}\circ -\right) \cong J_{X, Y}$.
\end{itemize}

\noindent The result will then follow in Theorem \ref{Two dimensional Abstract Kleisli equivalent characterisations} by the two-out-of-three property for equivalences of categories. We begin by describing the thunking $2$-cell of $\underline{J}_{X, Y}\left(g, \bar{g}\right)$ and proving that it is a $2$-cell of free pseudoalgebras.

\begin{lemma}\label{theta of Gg}
	(Appendix \ref{proof Lemma theta of Gg}) Let $\left(\mathcal{B}, T\right)$ be a pseudomonad and $\left(g, \bar{g}\right) \in \mathbf{Cone}_{T}\left(X, Y\right)$. Then the pasting in the $2$-category $\mathcal{B}$ depicted below left is a $2$-cell of free $T$-pseudoalgebras as depicted below right.
	
	$$\begin{tikzcd}[column sep = 15, font=\fontsize{9}{6}]
		&&{T^2}Y
		\arrow[rr,"{\mu}_Y"]
		\arrow[rrdd,"{T^2}{\eta}_Y" description]
		\arrow[dd,Leftarrow,shorten=10,"T{\bar{g}}"']
		&&
		TY
		\arrow[rrdd,"T{\eta}_Y"]
		\arrow[dd,Leftarrow,shorten=10,"{\mu}_{{\eta}_Y}"]
		&
		{}
		\\
		\\
		TX
		\arrow[rruu,"Tg"]
		\arrow[rr,"Tg" description]
		\arrow[rrdd,"T{\eta}_X"']
		&&
		{T^2}Y
		\arrow[rr,"T{\eta}_{TY}" description]
		\arrow[rrrr,bend right = 42,"{1}_{{T^2}Y}" description]
		\arrow[rrdd,"T{\eta}_{TY}"description]
		\arrow[dd,Leftarrow,shorten=10,"T{\eta}_g"']
		&&
		{T^3}Y
		\arrow[rr,"{\mu}_{TY}" description]
		\arrow[d,Leftarrow,"{\rho}_{TY}",shorten=3]
		&&
		{T^2}Y&{}
		\\
		&&&&
		{}
		\arrow[d,Leftarrow,"T{\lambda}_Y",shorten=3]
		\\
		&&
		{T^2}X
		\arrow[rr,"{T^2}g"']
		&&
		{T^3}Y
		\arrow[rruu,"T{\mu}_Y"']
	\end{tikzcd}\begin{tikzcd}[column sep = 15, font=\fontsize{9}{6}]
		\left(TX{,}{\mu}_X\right)
		\arrow[rr,"\left(Tg{,}{\mu}_g\right)"]
		\arrow[dd,"\left(T{\eta}_X{,}{\mu}_{{\eta}_X}\right)"description]
		&&
		\left({T^2}Y{,}{\mu}_{TY}\right)
		\arrow[rr,"\left({\mu}_Y{,}{\alpha}_Y\right)"]
		\arrow[dd,shorten=10,Leftarrow]
		&&
		\left(TY{,}{\mu}_Y\right)
		\arrow[dd,"\left(T{\eta}_Y{,}{\mu}_{{\eta}_Y}\right)" description]
		\\
		\\
		\left({T^2}X{,}{\mu}_{TX}\right)
		\arrow[rr,"\left({T^2}g{,}{\mu}_{Tg}\right)"']
		&&
		\left({T^3}Y{,}{\mu}_{{T^2}Y}\right)
		\arrow[rr,"\left(T{\mu}_Y{,}{\mu}_{{\mu}_Y}\right)"']
		&&
		\left({T^2}Y{,}{\mu}_{TY}\right)
	\end{tikzcd}$$
\end{lemma}

\begin{proposition}\label{cones to theta pseudonat functor}
	(Appendix \ref{proof Proposition cones to theta pseudonat functor}) Let $\left(\mathcal{B}, T\right)$ be a pseudomonad. There is a functor $\underline{J}: \mathbf{Cone}_{T}\left(X, Y\right) \rightarrow \left({\B}_{T}\right)_{\theta}\big(\left(TX, \mu_{X}\right), \left(TY, \mu_{Y}\right)\big)$ which sends $\left(g, \bar{g}\right)$ to the $1$-cell whose underlying pseudomorphism is given by $\left(\mu_{Y}, \alpha_{y}\right)\circ\left(Tg, \mu_{g}\right)$ and whose thunking $2$-cell is described in Lemma \ref{theta of Gg}.
\end{proposition}

\begin{lemma}\label{cones to free theta pseudonat essentially surjective}
	(Appendix \ref{proof Lemma cones to free theta pseudonat essentially surjective}) Let $\left(\left(p, \bar{p}\right), \theta_{p}\right): \left(TX, \mu_{X}\right) \rightarrow \left(TY, \mu_{Y}\right)$ be a morphism in $\left({\B}_{T}\right)_{\theta}$. Then $p.\eta_{X}$ equipped with the following $2$-cell defines a isobidescent cone from $X$ to $Y$.
	
	$$\begin{tikzcd}[font=\fontsize{9}{6}, row sep = 15]
		X
		\arrow[rr,"{\eta}_X"]
		\arrow[dd,"{\eta}_X"']
		&
		{}
		\arrow[dd,shorten=10,Rightarrow,"{\eta}_{{\eta}_X}"]
		&
		TX
		\arrow[rr,"p"]
		\arrow[dd,"T{\eta}_X"]
		&
		{}
		\arrow[dd,Rightarrow,shorten=10,"{\theta}_p"]
		&
		TY
		\arrow[dd,"T{\eta}_Y"]
		\\
		\\
		TX
		\arrow[rr,"{\eta}_{TX}"']
		\arrow[rrdd,bend right,"p"']
		&
		{}
		&
		{T^2}X
		\arrow[rr,"Tp"']
		\arrow[dd,Rightarrow,shorten=10,"{\eta}_{p}"]
		&
		{}
		&
		{T^2}Y
		\\
		\\
		&&
		TY
		\arrow[rruu,bend right,"{\eta}_{TY}"']
	\end{tikzcd}$$
\end{lemma}

\begin{proposition}\label{cones to free theta pseudonat an equivalence}
	(Appendix \ref{proof cones to free theta pseudonat an equivalence}) The functor $\underline{J}_{X, Y}: \mathbf{Cone}_{T}\left(X, Y\right) \rightarrow \left({\B}_{T}\right)_{\theta}\big(\left(TX, \mu_{X}\right), \left(TY, \mu_{Y}\right)\big)$ is an equivalence of categories.
\end{proposition}

\begin{proposition}\label{G isomorphic to restriction along eta}
	(Appendix \ref{Proof of Proposition G isomorphic to restriction along eta}) Let $g: X \rightarrow Y$ be a morphism in $\B$. Then
	
	\begin{enumerate}
		\item  $\rho_{Y}: \underline{J}\left(\eta_{Y}g, \eta_{\eta_{Y}}\right) \rightarrow J\left(g\right)$ is a thunkable $2$-cell of free pseudoalgebras.
		\item $\rho_{Y}$ is the component at $g$ of a natural isomorphism
		
		$$\begin{tikzcd}[row sep = 15,font=\fontsize{9}{6}]
			\mathcal{B}\left(X{,}Y\right)
			\arrow[rr,"{\left({\eta}_Y{,}{\eta}_{{\eta}_Y}\right)}\circ{-}"]
			\arrow[rdd,"J_{X{,}Y}"']
			&&
			{\mathbf{Cone}_T}\left(X{,}Y\right)
			\arrow[ldd,"{\underline{J}}_{X{,}Y}"]
			\\
			&
			\cong
			\\
			&
			{{\left({\B}_T\right)}_{\theta}}{\left(X{,}Y\right)}
		\end{tikzcd}$$
	\end{enumerate}
\end{proposition}

\section{Main Results}\label{Section Main Results}

\noindent We now have all the ingredients to state and prove the $2$-categorical analogue of Theorem \ref{Conditions characterising monads whose EM adjunctions are of codescent type}.

\begin{theorem}\label{Two dimensional Abstract Kleisli equivalent characterisations}
	Let $\left(\mathcal{B}, T\right)$ be a pseudomonad, $\mathcal{B}^T$ its Eilenberg-Moore object and ${\left(\mathcal{B}^T\right)}^{\bar{T}}$ the coEilenberg-Moore object of the induced pseudocomonad on $\mathcal{B}^T$. Similarly, let ${\left({\B}_{T}\right)}^{\underline{T}}$ denote the coEilenberg-Moore object of the pseudocomonad induced on ${\B}_{T}$. Then the following are equivalent.
	
	\begin{enumerate}
		\item $J: \mathcal{B} \rightarrow \left({\B}_{T}\right)_{\theta}$ is a biequivalence.
		\item $\left(\mathcal{B}, T\right)$ satisfies isobidescent.
		\item The left pseudoadjoint $F_{T}: \mathcal{B} \rightarrow {\B}_{T}$ is faithful on $2$-cells, full on thunkable $2$-cells and surjective on $1$-cells which admit a thunking.
		\item The canonical comparison $2$-functor $\underline{K}: \mathcal{B} \rightarrow {\left({\B}_{T}\right)}^{\underline{T}}$ is bi-fully-faithful.
		\item The canonical comparison $2$-functor $\bar{K}: \mathcal{B} \rightarrow {\left(\mathcal{B}^T\right)}^{\bar{T}}$ is bi-fully-faithful. 
	\end{enumerate}
\end{theorem}

\begin{proof}

	$(1) \iff (2)$ follows from Proposition \ref{G isomorphic to restriction along eta} by the two-out-of-three property for equivalences of categories given that by Proposition \ref{cones to free theta pseudonat an equivalence} $\bar{G}_{X, Y}$ is an equivalence and $G$ is bijective on objects. For $(1) \iff (3)$ observe $F_T = F_{\theta}.J$ and $F_{\theta}$ is faithful on $2$-cells, so this holds for $F_{T}$ if and only if it holds for $J$. Meanwhile, $F_{T}$ being full on thunkable $2$-cells (resp. essentially surjective on $1$-cells which admit a thunking) is clearly equivalent to $J$ being full on $2$-cells (resp. essentially surjective on $1$-cells which admit a thunking), since these properties characterise the $1$ and $2$-cells which are in ${\left({\B}_{T}\right)}_{\theta}$. The equivalence $\left(2\right) \iff \left(5\right)$ is true since by Proposition \ref{theta pseudonat 2-category from abstract kleisli structure} part (7), the $2$-category ${\left({\B}_{T}\right)}_{\theta}$ is precisely the image of the canonical comparison $\underline{K}: \mathcal{B} \rightarrow {\left({\B}_{T}\right)}^{\underline{T}}$. Finally, for $(4) \iff (5)$ one observes that the image under $\bar{K}: \mathcal{B} \rightarrow {\left(\mathcal{B}^T\right)}^{\bar{T}}$ of every $X \in \mathcal{B}$ will have an underlying $T$-pseudoalgebra which is free on $X$. Hence $\bar{K}$ will factorise through $\underline{K}$ via a $2$-fully-faithful $2$-functor, since ${\B}_{T} \rightarrow \mathcal{B}^T$ is $2$-fully-faithful. 
\end{proof}

\begin{definition}\label{trikleisli extensions}
	We describe a $\mathbf{Gray}$-category that can be seen as categorifying $\mathbf{KL}\left(\mathbf{Cat}\right)$, and whose description is dual of the description of $\mathbf{EM}\left(\mathcal{K}\right)$ given in section 2.2 of \cite{FTM2}. It will be denoted as ${\mathbf{KLExt}\left(\mathbf{Gray}\right)}_{\lambda}$, and has data as described below. There is also the $\mathbf{Gray}$-category $\mathbf{KLExt}\left(\mathbf{Gray}\right)_{\tau}$, described following Corollary 4.6 of \cite{Formal Theory of Pseudomonads}, and an identity on objects and arrows $\mathbf{Gray}$-functor
	\\ $\mathbf{KLExt}\left(\mathbf{Gray}\right)_{\tau} \rightarrow \mathbf{KLExt}\left(\mathbf{Gray}\right)_{\lambda}$.
	\begin{itemize}
		\item Objects given by pseudomonads $\left(\A, S\right)$.
		\item Morphisms $\left(G, \bar{G}\right): \left(\A, S\right) \rightarrow \left(\B, T\right)$ given by pairs of $2$-functors $G: \A \rightarrow \B$ and $\bar{G}: \A_S \rightarrow \B_T$ satisfying $\bar{G}.F_{S} = F_{T}.G$, where $F_{S}$ and $F_{T}$ are the left pseudoadjoints to the $2$-categories of free pseudoalgebras.
		\item $2$-cells $\phi: \left(G, \bar{G}\right) \Rightarrow \left(H, \bar{H}\right)$ given by arbitrary pseudonatural transformations $\phi: \bar{G} \Rightarrow \bar{H}$.
		\item $3$-cells $\Omega: \phi \Rrightarrow \psi$ are given by arbitrary modifications with source $\phi$ and target $\psi$.
	\end{itemize}
\end{definition}

\noindent We now turn to showing that $J$ defines the unit of reflections from $\mathbf{Gray}$-categories $\mathbf{KLExt}\left(\mathbf{Gray}\right)_{\kappa}$, to $\mathbf{Gray}$-categories of $2$-abstract Kleisli structures.

\begin{definition}\label{Gray-cat of Abs KL of 2-cats}
	Let $2$-$\mathbf{AbsKL}_{0}$ be the class of $2$-abstract Kleisli structures. The $\mathbf{Gray}$-category $2$-$\mathbf{AbsKL}_\tau$ will be defined as the intermediate $\mathbf{Gray}$-category appearing in the bijective on objects/fully-faithful factorisation of the assignment $2$-$\mathbf{AbsKL}_{0} \rightarrow \mathbf{KLExt}\left(\mathbf{Gray}\right)_\tau$ which sends the abstract Kleisli structure $\left(\mathcal{B}, T, \pi\right)$ to the pseudomonad $\left(\mathcal{B}_\pi, U_\pi.F_\pi\right)$ induced by the pseudoadjunction described in Proposition \ref{theta pseudonat 2-category from abstract kleisli structure} part 7. The $\mathbf{Gray}$-category $2$-$\mathbf{AbsKL}_\lambda$ will similarly be defined as the bijective on objects/fully-faithful factorisation of the same assignment this time viewed as $2$-$\mathbf{AbsKL}_{0} \rightarrow \mathbf{KLExt}\left(\mathbf{Gray}\right)_\lambda$.
\end{definition}

\noindent We will need Lemmas \ref{Bar G p bar p thunkable in B pi} and \ref{pseudonaturality constraint thunkable in B pi} to prove the desired universal property of $\left(1_{{\A}_{S}}, J\right)$. We fix the following notation \begin{itemize}
	\item $\left(A, S, \eta, \mu, \lambda, \alpha, \rho\right)$ is a pseudomonad. 
	\item $\left(\left(p, \bar{p}\right), \theta_{p}\right): X \rightarrow Y$ is a $1$-cell in ${\left({\A}_{S}\right)}_{\theta}$.
	\item $\left(\mathcal{B}, T, \pi\right)$ is an abstract Kleisli structure on a $2$-category $\mathcal{B}$.
	\item  $\left(G, \bar{G}\right)$ and $\left(H, \bar{H}\right)$ are morphisms of $2$-abstract Kleisli structures from $ {\left({\A}_{S}, \underline{S}, \theta\right)}$ to $ \left(\mathcal{B}, T, \pi\right)$.
	\item $\left(\phi, \bar{\phi}\right): \left(G, \bar{G}\right) \Rightarrow \left(H, \bar{H}\right)$ and $\left(\psi, \bar{\psi}\right): \left(G, \bar{G}\right) \Rightarrow \left(H, \bar{H}\right)$ are tight $2$-cells of $2$-abstract Kleisli structures.
	\item $\Omega: \left(\phi, \bar{\phi}\right) \Rrightarrow \left(\psi, \bar{\psi}\right)$ is a tight $3$-cell of $2$-abstract Kleisli structures.
\end{itemize}

\begin{lemma}\label{Bar G p bar p thunkable in B pi}
	The morphism $\bar{G}\left(p, \bar{p}\right)$ in $\mathcal{B}$ has a thunking given by the following pasting in $\mathcal{B}$. 
	
	$$\begin{tikzcd}[font=\fontsize{9}{6}, row sep = 15]
		&&
		GSX
		\arrow[rrrr,"{\bar{G}}{\left(p{,}\bar{p}\right)}"{name=D}]
		\arrow[dd,"G\eta_X"]
		\arrow[llddd,bend right=20,"{\pi}_{GSX}"'{name=A}]
		&&&&
		GSY
		\arrow[rrddd,bend left=20,"{\pi}_{GSY}"{name=H}]
		\arrow[dd,"G\eta_Y"']
		\\
		\\
		&&
		G{S^2}X
		\arrow[rrrr,"Gp"{name=E}]
		\arrow[dd, "\pi_{GS^2 X}"]
		\arrow[dll, Rightarrow, shorten = 10, "\pi_{G\eta_{X}}"]
		&&&&
		G{S^2}Y
		\arrow[dd, "\pi_{GS^2 X}"']
		\\
		TGSX
		\arrow[rrd, "TG\eta_{X}"{name=B}]
		\arrow[rrddd, bend right=20, "1_{TGX}"'{name=C}]
		&&&&&&&&
		TGSY
		\arrow[lldd, Rightarrow, shorten = 35, "T\bar{G}\rho_{Y}"]
		\arrow[lld, "TG\eta_{Y}"'{name=I}]
		\arrow[llddd,bend left=20, "1_{TGSY}"{name=J}]
		\\
		&&
		TG{S^2}X
		\arrow[rrrr, "TGSp"{name=F}]
		\arrow[dd, "T\bar{G}{\left(\mu_X{,}\alpha_X\right)}_{X}" description]
		&&&&
		TG{S^2}Y
		\arrow[dd, "T\bar{G}{\left(\mu_Y{,}\alpha_Y\right)}_{Y}"' description]
		\\
		&&&&&&{}
		\\
		&&
		TGSX
		\arrow[rrrr, "T\bar{G}\left(p{,}\bar{p}\right)"'{name=G}]
		&&&&
		TGSY
		\arrow[from=B,to=C,Rightarrow,shorten=15,"{T}{\bar{G}}{{\rho}_X}^{-1}"]
		\arrow[from=D,to=E,Rightarrow,shorten=10,"{\bar{G}}{{\theta}_p}"]
		\arrow[from=E,to=F,Rightarrow,shorten=15,"{\pi}_{Gp}"]
		\arrow[from=F,to=G,Rightarrow,shorten=15,"{T}{\bar{G}}{\bar{p}}"]
		\arrow[from=H,to=I,Rightarrow,shorten=15,"{\pi}_{G{\eta}_Y}"']
	\end{tikzcd}$$
\end{lemma}

\begin{proof}
	This follows via similar techniques to those in other proofs. We omit details as they are significantly more tedious, but refer the interested reader to Appendix 11.5 and the proof of Lemma 8.2.14 of \cite{Miranda PhD}.
\end{proof}

\begin{corollary}\label{G' unique 2-functor}
	\hspace{1mm}
	\begin{enumerate}
		\item The assignment which sends the free pseudoalgebra $\left(SX, \mu_{X}\right)$ to $GX$, the thunked pseudomorphism $\left(\left(p, \bar{p}\right), \theta_{p}\right): \left(SX, \mu_{X}\right) \rightarrow \left(SY, \mu_{Y}\right)$ to $\bar{G}\left(p, \bar{p}\right)$ equipped with the thunking given by the $2$-cell described in Lemma \ref{Bar G p bar p thunkable in B pi}, and the thunkable $2$-cell $\chi: \left(\left(p, \bar{p}\right), \theta_{p}\right) \Rightarrow \left(\left(q, \bar{q}\right), \theta_{q}\right)$ to $\bar{G}\chi$ extends to a $2$-functor $G': {\left({\B}_{T}\right)}_{\theta} \rightarrow \mathcal{B}_\pi$.
		\item $G'$ is the unique $2$-functor satisfying \begin{enumerate}
			\item $G'J = G$, and
			\item $F_\pi .G' = \bar{G}.F_\theta$.
		\end{enumerate}
	\end{enumerate}
	
\end{corollary}

\begin{proof}
	\noindent For part 1, first observe that if $\chi: \left(\left(p, \bar{p}\right), \theta_{p}\right) \Rightarrow \left(\left(q, \bar{q}\right), \theta_{q}\right)$ is a thunkable $2$-cell in $\mathbf{FreePsAlg}_{S}$ then $\bar{G}\chi$ is also a thunkable $2$-cell in $\mathcal{B}$. This follows from the thunkability condition for $\chi$, pseudonaturality of $\pi$ on $G\chi$, and the coherence for $\chi$ as a $2$-cell of pseudoalgebras. Then functoriality of $G'$ between hom categories is clear, while $2$-functoriality of $G'$ follows from that of $\bar{G}$, pseudonaturality of $\pi$, and by cancelling components of $T\bar{G}\rho$ and $\pi_{G\eta}$ with their inverses.
	\\
	\\
	For part $2$, $G'$ satisfies condition $(b)$ by construction. To see that it also satisfies condition $(a)$, it suffices to consider the thunking described in Lemma \ref{Bar G p bar p thunkable in B pi} in the case where $\left(\left(p, \bar{p}\right), \theta_{p}\right) = \left(\left(Sf, \mu_{f}\right), S\eta_{f}\right)$ and observe that this simplifies to $\pi_{Gf}$. This uses pseudonaturality of $\pi$ on $G\eta_{f}$ and the modification coherence for $\rho$ on $f$. Finally for uniqueness, observe that $\bar{G}.F_{\theta} = F_\pi .G''$ implies that $G''\chi = \bar{G}\chi$ for any $2$-functor $G''$, and that since $G''$ agrees with $G'$ on $2$-cells it must equal $G'$.
\end{proof}

\begin{lemma}\label{pseudonaturality constraint thunkable in B pi}
	Let $\left(\phi, \bar{\phi}\right): \left(G, \bar{G}\right) \Rightarrow \left(H, \bar{H}\right)$ be a tight $2$-cell of $2$-abstract Kleisli structures. Then the pseudonaturality component $\bar{\phi}_{\left(p, \bar{p}\right)}$ is a thunkable $2$-cell in $\mathcal{B}$.
\end{lemma}

\begin{proof}
	This is proved using thunkability of $\phi_{\eta_{X}}$ and $\phi_{\eta_{Y}}$, and pseudonaturality of $\overline{\phi}$, as detailed in Appendix 11.6 of \cite{Miranda PhD}.
\end{proof}

\begin{corollary}\label{phi' unique factorisation through J}
	Let $G'$ be as defined in \ref{G' unique 2-functor} and let $H'$ be defined analogously from $\left(H, \bar{H}\right)$. Then
	
	\begin{enumerate}
		\item There is a pseudonatural transformation $\phi': G' \Rightarrow H'$ with component at $X$ given by $\phi_{X}$ and component at $\left(\left(p, \bar{p}\right), \theta_{p}\right)$ given by $\bar{\phi}_{\left(p, \bar{p}\right)}$.
		\item $\phi'$ is the unique pseudonatural transformation satisfying \begin{enumerate}
			\item $\phi'J=\phi$, and
			\item $\bar{\phi}.F_\theta = F_\pi . \phi'$.
		\end{enumerate}
	\end{enumerate} 
\end{corollary}

\begin{proof}
	For part (1), the conditions for pseudonaturality of $\phi'$ follow directly from the analogous conditions for $\bar{\phi}$. For part (2) it is clear from the definition of $\phi'$ and the fact that $F_\pi.\phi = \bar{\phi}.F_\theta$ that $\phi'$ uniquely satisfies conditions $(a)$ and $(b)$.
\end{proof}

\begin{theorem}\label{J a unit for 2 abstract kleisli structures}
	The inclusion $I: 2$-$\mathbf{AbsKL} \rightarrow \mathbf{KLExt}\left(\mathbf{Gray}\right)$ has a left $\mathbf{Gray}$-adjoint which sends $\left(\A, S\right)$ to the $2$-abstract Kleisli structure $\left({\A}_{S}, \underline{S}, F_{S}\eta \right)$, and the unit of this adjunction at $\left(\A, S\right)$ is given by $\left(J, 1_{{\A}_{S}}\right): \left(\A, S\right) \rightarrow \left({\left({\A}_{S}\right)}_{\theta}, S'\right)$.
\end{theorem}

\begin{proof}
	It suffices to show that the $2$-functor depicted below, which is induced by precomposition along $\left(J, 1_{\mathbf{FreePsAlg_{S}}}\right)$, is an isomorphism of $2$-categories.
	\\
	$$\mathbf{KLExt}\left(\mathbf{Gray}\right)\left(\left(\left({\A}_{S}\right)_{\theta}{,}\underline{S}'\right){,}\left(\mathcal{B}_{\pi}{,}T'\right)\right) \rightarrow \mathbf{KLExt}\left(\mathbf{Gray}\right)\left(\left(\A{,} S\right){,}\left(\mathcal{B}_{\pi}{,}T'\right)\right)$$
	\\
	In Corollaries \ref{G' unique 2-functor} and \ref{phi' unique factorisation through J} we have already seen that the actions of this $2$-functor on objects and on morphisms are bijections. Let $\left(\Omega, \bar{\Omega}\right): \left(\phi, \bar{\phi}\right) \Rrightarrow \left(\psi, \bar{\psi}\right)$ be a $3$-cell of $2$-abstract Kleisli structures. Then observe that $\bar{\Omega}_{X} = \Omega_{X}$ for every $X \in \A$, and hence the modification coherence for $\bar{\Omega}$ implies that $X \mapsto \Omega_{X}$ also extends to a modification $\Omega': \phi' \Rrightarrow \psi'$. Finally, observe that $\Omega'$ is indeed the unique modification $\phi' \Rrightarrow \psi'$ satisfying $\Omega'J = \Omega$ and $F_\pi . \Omega = \bar{\Omega}.F_\theta$. All of these observations are straightforward since the $2$-functors $F_\pi$, $F_\theta$ and $J$ are all bijective on objects. This completes the proof.
\end{proof}

\section{Concluding remarks}\label{conclusion}

\noindent Führmann showed that abstract Kleisli structures form a full reflective sub-category of the category of monads and strict morphisms, whose essential image consists of those monads whose unit $\eta: 1_{\A} \Rightarrow S$ is the equaliser of $S\eta$ and $\eta_{S}$. We have further abstracted these structures, and shown that Führmann's adjunction underlies a reflective $2$-adjunction. We then defined abstract Kleisli structures on $2$-categories, and proven two-dimensional analogues of both Führmann's results and our extensions of those results. Specifically, we have shown that abstract Kleisli structures on $2$-categories can be seen as the objects of full sub-$\mathbf{Gray}$-categories of either ${\mathbf{KLExt}\left(\mathbf{Gray}\right)}_{\tau}$ or $\mathbf{KLExt}\left(\mathbf{Gray}\right)_{\lambda}$, with these structures being described in following Corollary 4.6 in \cite{Formal Theory of Pseudomonads} and in Definition \ref{trikleisli extensions}, respectively. In both instances, there is a reflection to the inclusion given by passing to the $2$-category whose morphisms are equipped with thunkings, and whose $2$-cells are thunkable. If the data $\left(\eta, \eta_{\eta}\right)$ extracted from a given pseudomonad is a certain isobidescent cone, then that pseudomonad is biequivalent to the canonical pseudomonad formed by an abstract Kleisli structure on a $2$-category. The base of this pseudomonad consists of morphisms equipped with thunkings, and the thunkable $2$-cells. Also equivalent to this bicategorical limit condition are certain criteria on the left pseudoajoint, or bi-fully faithfulness of comparisons from the base $2$-category to $2$-categories of descent data.

\section{Appendices of proofs}\label{Appendix}

\begin{notation}
	We use colour to draw the reader's attention to new data appearing in each step of a proof involving a pasting diagram chase. In particular, we use \textcolor{blue}{blue} for new objects and morphisms and \textcolor{red}{red} for new $2$-cells. To avoid clutter, we omit denoting inverses of $2$-cells with $(-)^{-1}$. The reader should be able to infer from the source and target of an invertible $2$-cell denoted $\gamma$ if it is actually the inverse $\gamma^{-1}$.
\end{notation}

\subsection{Proof of Proposition \ref{theta pseudonat 2-category from abstract kleisli structure}}\label{proof Proposition theta pseudonat 2-category from abstract kleisli structure}

 We first observe that the $2$-category $\mathcal{B}_{\theta}$ has
 
 \begin{itemize}
 	\item Objects the same as $\mathcal{B}$.
 	\item Arrows $\left(f, \theta_{f}\right): X \rightarrow Y$ consisting of an arrow $f: X \rightarrow Y$ in $\mathcal{B}$ and an invertible $2$-cell $\theta_{f}: \theta_{Y}.f \Rightarrow Qf.\theta_{X}$ satisfying the following equations.
 	
 	\begin{itemize}
 		\item (Unit condition)
 		\\
 		\begin{tikzcd}[font=\fontsize{9}{6}]
 			X
 			\arrow[rr,"f"]
 			\arrow[dd,"{\theta}_X"']
 			&
 			{}
 			\arrow[dd,Rightarrow,shorten=15,"{\theta}_f"]
 			&
 			Y
 			\arrow[dd,"{\theta}_Y"]
 			\arrow[rrdd,bend left,"1_Y"]
 			&
 			{}
 			\arrow[dd,Rightarrow,shorten=15,"u_Y"]
 			\\
 			\\
 			QX
 			\arrow[rr,"Qf"']
 			&
 			{}
 			&
 			QY
 			\arrow[rr,"{\varepsilon}_Y"']
 			&
 			{}
 			&
 			Y
 		\end{tikzcd} = \begin{tikzcd}[font=\fontsize{9}{6}]
 			X
 			\arrow[rr,"1_X"]
 			\arrow[dd,"{\theta}_X"']
 			&
 			{} \arrow[d,Rightarrow,"u_X"',shorten=4,shift right=3]
 			&
 			X
 			\arrow[dd,Rightarrow,shorten=12,"{\varepsilon}_f"]
 			\arrow[rrdd,bend left,"f"]
 			\\
 			&
 			{}
 			\\
 			QX
 			\arrow[rr,"Qf"']
 			\arrow[rruu,"{\varepsilon}_X"']
 			&
 			{}
 			&
 			QY
 			\arrow[rr,"{\varepsilon}_Y"']
 			&&
 			Y
 		\end{tikzcd}
 		\item (Associativity condition)
 		
 	\end{itemize}
 	\noindent\begin{tikzcd}[column sep = 15, font=\fontsize{9}{6}]
 		&&X
 		\arrow[lldd, "\theta_{X}"']
 		\arrow[rrdd, "\theta_{X}"]
 		\arrow[dddd, Rightarrow, "m_{X}", shorten = 30]
 		\arrow[rr, "f"]
 		&&Y
 		\arrow[dd, Rightarrow, "\theta_{f}", shorten = 10]
 		\arrow[rrdd, "\theta_{Y}"]
 		\\
 		\\
 		QX\arrow[rrdd, "Q\theta_{X}"']
 		&&&&QX
 		\arrow[lldd, "\delta_{X}"]
 		\arrow[rr, "Qf"]
 		\arrow[dd, Rightarrow, "\delta_{f}", shorten = 10]
 		&&QY\arrow[lldd, "\delta_{Y}"]&=
 		\\
 		\\
 		&&Q^2 X
 		\arrow[rr, "Q^2 f"']
 		&&Q^2Y
 	\end{tikzcd}\begin{tikzcd}[column sep = 15, font=\fontsize{9}{6}]
 		&&X
 		\arrow[lldd, "\theta_{X}"']
 		\arrow[dd, Rightarrow, "\theta_{f}", shorten = 10]
 		\arrow[rr, "f"]
 		&&Y
 		\arrow[lldd, "\theta_{Y}"]
 		\arrow[dddd, Rightarrow, "m_{Y}", shorten = 30]
 		\arrow[rrdd, "\theta_{Y}"]
 		\\
 		\\
 		QX
 		\arrow[rrdd, "Q\theta_{X}"']
 		\arrow[rr, "Qf"]
 		&&QY
 		\arrow[dd, Rightarrow, "Q\theta_{f}", shorten = 10]
 		\arrow[rrdd, "Q\theta_{Y}"]
 		&&
 		&&QY
 		\arrow[lldd, "\delta_{Y}"]
 		\\
 		\\
 		&&Q^2 X
 		\arrow[rr, "Q^2 f"']
 		&&Q^2Y
 	\end{tikzcd}
 	
 	\item $2$-cells $\phi: \left(f, \theta_{f}\right) \rightarrow \left(g, \theta_{g}\right)$ given by $2$-cells $\phi: f \Rightarrow g$ in $\mathcal{B}$ satisfying the following equation.
 \end{itemize}
 $$\begin{tikzcd}[font=\fontsize{9}{6}]
 	X
 	\arrow[rr,bend left,"f" name=A]
 	\arrow[rr,bend right,"g"' name=B]
 	\arrow[dd,"{\theta}_X"']
 	&&
 	Y
 	\arrow[dd,"{\theta}_Y"]
 	\\
 	\\
 	QX \arrow[rr,bend right,"Qg"'name=C]
 	&&
 	QY
 	\arrow[from=A,to=B,Rightarrow,"\phi",shorten=8]
 	\arrow[from=B,to=C,Rightarrow,"{\theta}_g",shorten=15]
 \end{tikzcd} = \begin{tikzcd}[font=\fontsize{9}{6}]
 	X
 	\arrow[rr,bend left,"f" name=A]
 	\arrow[dd,"{\theta}_X"']
 	&&
 	Y
 	\arrow[dd,"{\theta}_Y"]
 	\\
 	\\
 	QX
 	\arrow[rr,bend left,"Qf" name=B]
 	\arrow[rr,bend right,"Qg"'name=C]
 	&&
 	QY
 	\arrow[from=A,to=B,Rightarrow,"{\theta}_f",shorten=13]
 	\arrow[from=B,to=C,Rightarrow,"Q{\phi}",shorten=8]
 \end{tikzcd}$$

\noindent There is a $2$-functor $F: \mathcal{B}_\theta \rightarrow \mathcal{B}$ which forgets data of the form $\theta_{f}$, and is hence clearly bijective on objects and faithful on $2$-cells. Moreover, there is a $2$-functor $U: \mathcal{B} \rightarrow \mathcal{B}_{\theta}$ which sends $1$-cells $f$ to $\left(Qf, \delta_{f}\right)$ and takes the image under $Q$ on objects and $2$-cells. It follows from the equations displayed above that the assignment $X \mapsto \theta_{X}$ extends to a pseudonatural transformation $\theta: 1_{\mathcal{B}_{\theta}} \Rightarrow UF$ with component at $\left(f, \theta_{f}\right)$ given by $\theta_{f}$, and the assignment $X \mapsto u_{X}$ extends to an invertible modification. Finally, the right tetrahedral identity follows from the unit law for the pseudocoalgebra $\left(X, \theta_{X}, u_{X}, m_{X}\right)$, while the left tetrahedral identity is coherence 3 for the pseudocomonad $\left(\mathcal{B}, Q\right)$.

\subsection{Proof of Lemma \ref{theta of Gg}}\label{proof Lemma theta of Gg}

\begin{proof}
	Beginning with the diagram depicted below left, apply the modification coherence for $\alpha$ on $\eta_{Y}$ to get below right.
	\\
	\\
	
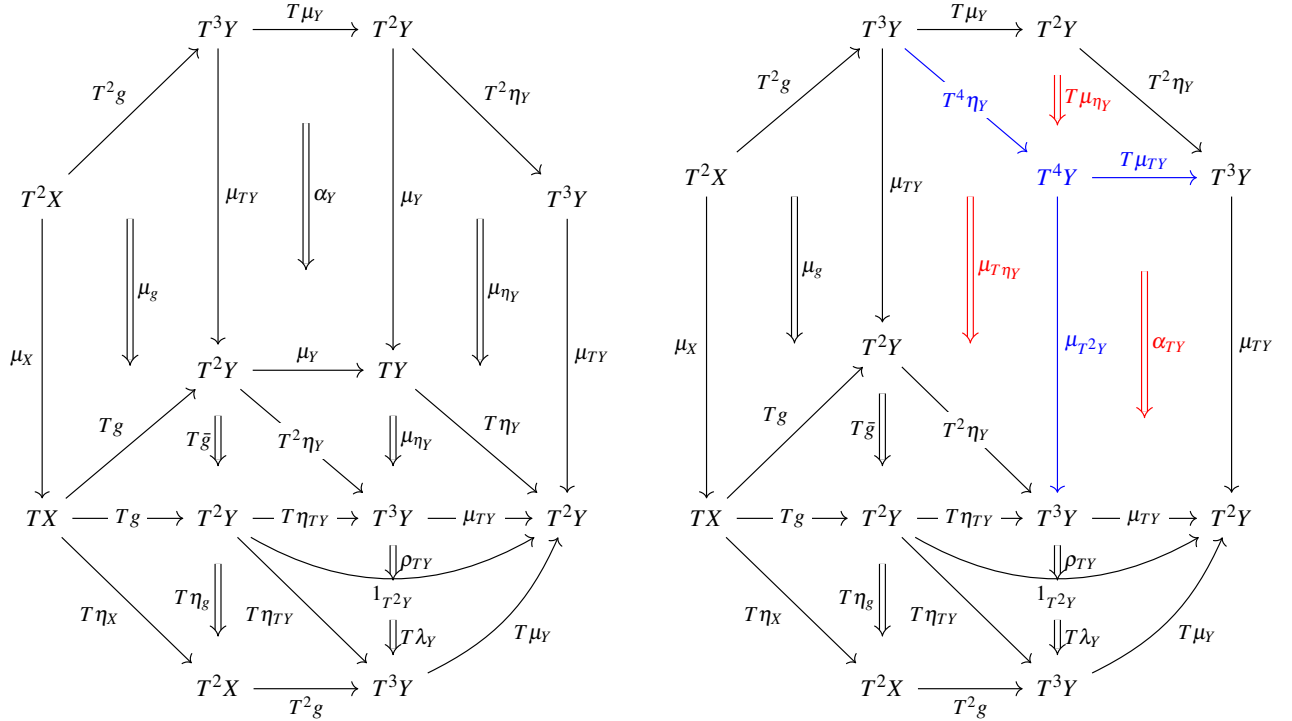
\noindent \begin{tikzcd}[column sep = 15, ,font=\fontsize{9}{6}]
		&&
		{T^3}Y
		\arrow[dddd,"{\mu}_{TY}"]
		\arrow[rr,"T{\mu}_Y"]
		&&
		{T^2}Y
		\arrow[dddd,"{\mu}_Y"]
		\arrow[rrdd,"{T^2}{\eta}_Y"]
		\\
		&&&
		{}
		\arrow[dd,Rightarrow,"{\alpha}_Y"]
		\\
		{T^2}X
		\arrow[rruu,"{T^2}g"]
		\arrow[dddd,"{\mu}_X"']
		&
		{}
		\arrow[dd,Rightarrow,"{\mu}_g"]
		&&&&
		{}
		\arrow[dd,Rightarrow,"{\mu}_{{\eta}_Y}"]
		&
		{T^3}Y
		\arrow[dddd,"{\mu}_{TY}"]
		\\
		&&&
		{}
		\\
		&
		{}
		&
		{T^2}Y
		\arrow[rr,"{\mu}_Y"]
		\arrow[rrdd,"{T^2}{\eta}_Y" description]
		\arrow[dd,Rightarrow,shorten=10,"T{\bar{g}}"']
		&&
		TY
		\arrow[rrdd,"T{\eta}_Y"]
		\arrow[dd,Rightarrow,shorten=10,"{\mu}_{{\eta}_Y}"]
		&
		{}
		\\
		\\
		TX
		\arrow[rruu,"Tg"]
		\arrow[rr,"Tg" description]
		\arrow[rrdd,"T{\eta}_X"']
		&&
		{T^2}Y
		\arrow[rr,"T{\eta}_{TY}" description]
		\arrow[rrrr,bend right,"{1}_{{T^2}Y}"']
		\arrow[rrdd,"T{\eta}_{TY}"']
		\arrow[dd,Rightarrow,shorten=10,"T{\eta}_g"']
		&&
		{T^3}Y
		\arrow[rr,"{\mu}_{TY}" description]
		\arrow[d,Rightarrow,"{\rho}_{TY}",shorten=3]
		&&
		{T^2}Y&{}
		\\
		&&&&
		{}
		\arrow[d,Rightarrow,"T{\lambda}_Y",shorten=3]
		\\
		&&
		{T^2}X
		\arrow[rr,"{T^2}g"']
		&&
		{T^3}Y
		\arrow[rruu,"T{\mu}_Y"', bend right = 20]
	\end{tikzcd}\begin{tikzcd}[column sep = 15, ,font=\fontsize{9}{6}]
		&&
		{T^3}Y
		\arrow[dddd,"{\mu}_{TY}"]
		\arrow[rr,"T{\mu}_Y"]
		\arrow[rrdd,"{T^4}{\eta}_Y" description,blue]
		&&
		{T^2}Y
		\arrow[rrdd,"{T^2}{\eta}_Y"]
		\arrow[dd,Rightarrow,shorten=10,red,"T{\mu}_{{\eta}_Y}"]
		\\
		&&&
		\\
		{T^2}X
		\arrow[rruu,"{T^2}g"]
		\arrow[dddd,"{\mu}_X"']
		&
		{}
		\arrow[dd,Rightarrow,"{\mu}_g"]
		&&
		{}
		\arrow[dd,Rightarrow,red,"{\mu}_{T{\eta}_Y}"]
		&
		\color{blue}{T^4}Y
		\arrow[rr,blue,"T{\mu}_{TY}"]
		\arrow[dddd,blue,"{\mu}_{{T^2}Y}"]
		&&
		{T^3}Y
		\arrow[dddd,"{\mu}_{TY}"]
		\\
		&&&&&
		{}
		\arrow[dd,Rightarrow,red,"{\alpha}_{TY}"]
		\\
		&
		{}
		&
		{T^2}Y
		\arrow[rrdd,"{T^2}{\eta}_Y" description]
		\arrow[dd,Rightarrow,shorten=10,"T{\bar{g}}"']
		&
		{}
		\\
		&&&&&
		{}
		\\
		TX
		\arrow[rruu,"Tg"]
		\arrow[rr,"Tg" description]
		\arrow[rrdd,"T{\eta}_X"']
		&&
		{T^2}Y
		\arrow[rr,"T{\eta}_{TY}" description]
		\arrow[rrrr,bend right,"{1}_{{T^2}Y}"']
		\arrow[rrdd,"T{\eta}_{TY}"']
		\arrow[dd,Rightarrow,shorten=10,"T{\eta}_g"']
		&&
		{T^3}Y
		\arrow[rr,"{\mu}_{TY}" description]
		\arrow[d,Rightarrow,"{\rho}_{TY}",shorten=3]
		&&
		{T^2}Y
		\\
		&&&&
		{}
		\arrow[d,Rightarrow,"T{\lambda}_Y",shorten=3]
		\\
		&&
		{T^2}X
		\arrow[rr,"{T^2}g"']
		&&
		{T^3}Y
		\arrow[rruu,"T{\mu}_Y"', bend right = 20]
	\end{tikzcd}
	\\
	\\
	\noindent Then apply pseudonaturality of $\mu$ on $\bar{g}$ to reduce above right to below left, followed by coherence $5$ for the pseudomonad $\left(A, T\right)$ to reduce below left to below right.
	\\
	\\
	
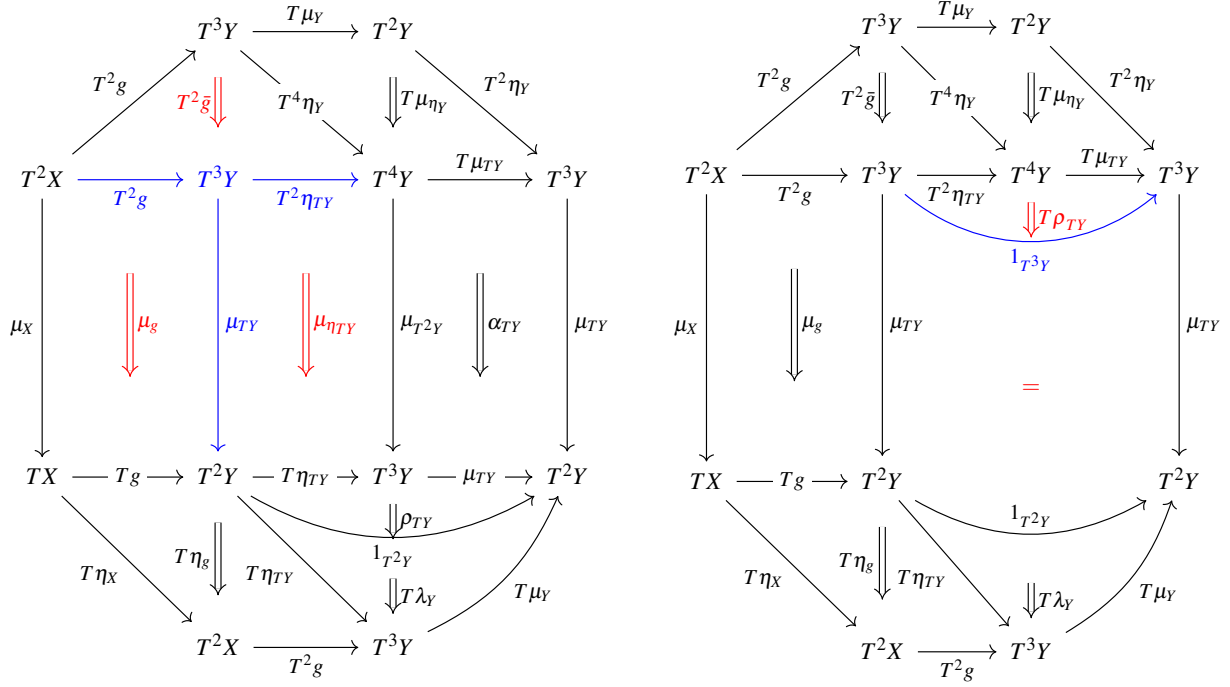
\noindent \begin{tikzcd}[column sep = 15, ,font=\fontsize{9}{6}]
		&&
		{T^3}Y
		\arrow[rr,"T{\mu}_Y"]
		\arrow[rrdd,"{T^4}{\eta}_Y" description]
		\arrow[dd,shorten=10,Rightarrow,"{T^2}{\bar{g}}"',red]
		&&
		{T^2}Y
		\arrow[rrdd,"{T^2}{\eta}_Y"]
		\arrow[dd,Rightarrow,shorten=10,"T{\mu}_{{\eta}_Y}"]
		\\
		&&&
		\\
		{T^2}X
		\arrow[rruu,"{T^2}g"]
		\arrow[dddd,"{\mu}_X"']
		\arrow[rr,blue,"{T^2}g"']
		&&
		\color{blue}{T^3}Y
		\arrow[rr,blue,"{T^2}{\eta}_{TY}"']
		\arrow[dddd,blue,"{\mu}_{TY}"]
		&&
		{T^4}Y
		\arrow[rr,"T{\mu}_{TY}"]
		\arrow[dddd,"{\mu}_{{T^2}Y}"]
		&&
		{T^3}Y
		\arrow[dddd,"{\mu}_{TY}"]
		\\
		&
		{}
		\arrow[dd,Rightarrow,"{\mu}_g",red]
		&&
		{}
		\arrow[dd,Rightarrow,"{\mu}_{{\eta}_{TY}}",red]
		&&
		{}
		\arrow[dd,Rightarrow,"{\alpha}_{TY}"]
		\\
		\\
		&
		{}
		&&
		{}
		&&
		{}
		\\
		TX
		\arrow[rr,"Tg" description]
		\arrow[rrdd,"T{\eta}_X"']
		&&
		{T^2}Y
		\arrow[rr,"T{\eta}_{TY}" description]
		\arrow[rrrr,bend right,"{1}_{{T^2}Y}"']
		\arrow[rrdd,"T{\eta}_{TY}"']
		\arrow[dd,Rightarrow,shorten=10,"T{\eta}_g"']
		&&
		{T^3}Y
		\arrow[rr,"{\mu}_{TY}" description]
		\arrow[d,Rightarrow,"{\rho}_{TY}",shorten=3]
		&&
		{T^2}Y&{}
		\\
		&&&&
		{}
		\arrow[d,Rightarrow,"T{\lambda}_Y",shorten=3]
		\\
		&&
		{T^2}X
		\arrow[rr,"{T^2}g"']
		&&
		{T^3}Y
		\arrow[rruu,"T{\mu}_Y"', bend right =20]
	\end{tikzcd}\begin{tikzcd}[column sep = 15, ,font=\fontsize{9}{6}]
		&&
		{T^3}Y
		\arrow[rr,"T{\mu}_Y"]
		\arrow[rrdd,"{T^4}{\eta}_Y" description]
		\arrow[dd,shorten=10,Rightarrow,"{T^2}{\bar{g}}"']
		&&
		{T^2}Y
		\arrow[rrdd,"{T^2}{\eta}_Y"]
		\arrow[dd,Rightarrow,shorten=10,"T{\mu}_{{\eta}_Y}"]
		\\
		&&&
		\\
		{T^2}X
		\arrow[rruu,"{T^2}g"]
		\arrow[dddd,"{\mu}_X"']
		\arrow[rr,"{T^2}g"']
		&&
		{T^3}Y
		\arrow[rr,"{T^2}{\eta}_{TY}"']
		\arrow[dddd,"{\mu}_{TY}"]
		\arrow[rrrr,"{1}_{{T^3}Y}"',blue,bend right = 40]
		&&
		{T^4}Y
		\arrow[rr,"T{\mu}_{TY}"]
		\arrow[d,Rightarrow,shorten=3,red,"{T\rho}_{TY}"]
		&&
		{T^3}Y
		\arrow[dddd,"{\mu}_{TY}"]
		\\
		&
		{}
		\arrow[dd,Rightarrow,"{\mu}_g"]
		&&&
		{}
		\\
		\\
		&
		{}
		&&&
		\color{red}=
		\\
		TX
		\arrow[rr,"Tg" description]
		\arrow[rrdd,"T{\eta}_X"']
		&&
		{T^2}Y
		\arrow[rrrr,bend right,"{1}_{{T^2}Y}"]
		\arrow[rrdd,"T{\eta}_{TY}"']
		\arrow[dd,Rightarrow,shorten=10,"T{\eta}_g"']
		&&&&
		{T^2}Y
		\\
		&&&&
		{}
		\arrow[d,Rightarrow,"T{\lambda}_Y",shorten=3]
		\\
		&&
		{T^2}X
		\arrow[rr,"{T^2}g"']
		&&
		{T^3}Y
		\arrow[rruu,"T{\mu}_Y"', bend right = 20]
	\end{tikzcd}
	\\
	\\
	Apply pseudonaturality of $\mu$ on $\lambda_{Y}$ to reduce above right to below left. Finally, apply pseudonaturality of $\mu$ on $\eta_{g}$ to reduce below left to below right and complete the proof.
	\\
	\\ 
	\noindent \begin{tikzcd}[column sep = 15, ,font=\fontsize{9}{6}]
		&&
		{T^3}Y
		\arrow[rr,"T{\mu}_Y"]
		\arrow[rrdd,"{T^4}{\eta}_Y" description]
		\arrow[dd,shorten=10,Rightarrow,"{T^2}{\bar{g}}"']
		&&
		{T^2}Y
		\arrow[rrdd,"{T^2}{\eta}_Y"]
		\arrow[dd,Rightarrow,shorten=10,"T{\mu}_{{\eta}_Y}"]
		\\
		&&&
		\\
		{T^2}X
		\arrow[rruu,"{T^2}g"]
		\arrow[dddd,"{\mu}_X"']
		\arrow[rr,"{T^2}g"']
		&&
		{T^3}Y
		\arrow[rr,"{T^2}{\eta}_{TY}"']
		\arrow[dddd,"{\mu}_{TY}"]
		\arrow[rrrr,"{1}_{{T^3}Y}"',bend right]
		\arrow[rrddd,blue,"{T^2}{\eta}_{TY}"' description]
		&&
		{T^4}Y
		\arrow[rr,"T{\mu}_{TY}"]
		\arrow[d,Rightarrow,shorten=3,"{T\rho}_{TY}"]
		&&
		{T^3}Y
		\arrow[dddd,"{\mu}_{TY}"]
		\\
		&
		{}
		\arrow[dd,Rightarrow,"{\mu}_g"]
		&&&
		{}
		\arrow[dd,Rightarrow,shorten=10,red,"{T^2}{\lambda}_Y"]
		\\
		&&&
		{}
		\arrow[dd,Rightarrow,red,"{\mu}_{{\eta}_{TY}}"]
		&&
		{}
		\arrow[dd,Rightarrow,red,"{\mu}_{{\mu}_Y}"]
		\\
		&
		{}
		&&&
		\color{blue}{T^4}Y
		\arrow[rruuu,blue,"{T^2}{\mu}_Y" description]
		\arrow[ddd,blue,"{\mu}_{{T^2}Y}"]
		\\
		TX
		\arrow[rr,"Tg" description]
		\arrow[rrdd,"T{\eta}_X"']
		&&
		{T^2}Y
		\arrow[rrdd,"T{\eta}_{TY}"']
		\arrow[dd,Rightarrow,shorten=10,"T{\eta}_g"']
		&
		{}
		&&
		{}
		&
		{T^2}Y&{}
		\\
		\\
		&&
		{T^2}X
		\arrow[rr,"{T^2}g"']
		&&
		{T^3}Y
		\arrow[rruu,"T{\mu}_Y"']
	\end{tikzcd}\begin{tikzcd}[column sep = 15, ,font=\fontsize{9}{6}]
		&&
		{T^3}Y
		\arrow[rr,"T{\mu}_Y"]
		\arrow[rrdd,"{T^4}{\eta}_Y" description]
		\arrow[dd,shorten=10,Rightarrow,"{T^2}{\bar{g}}"']
		&&
		{T^2}Y
		\arrow[rrdd,"{T^2}{\eta}_Y"]
		\arrow[dd,Rightarrow,shorten=10,"T{\mu}_{{\eta}_Y}"]
		\\
		&&&
		\\
		{T^2}X
		\arrow[rruu,"{T^2}g"]
		\arrow[dddd,"{\mu}_X"']
		\arrow[rr,"{T^2}g"']
		\arrow[rrddd,blue,"{T^2}{\eta}_X" description]
		&&
		{T^3}Y
		\arrow[rr,"{T^2}{\eta}_{TY}"']
		\arrow[rrrr,"{1}_{{T^3}Y}"',bend right]
		\arrow[rrddd,"{T^2}{\eta}_{TY}"' description]
		\arrow[ddd,Rightarrow,shorten=15,red,"{T^2}{\eta}_g"']
		&&
		{T^4}Y
		\arrow[rr,"T{\mu}_{TY}"]
		\arrow[d,Rightarrow,shorten=3,"{T\rho}_{TY}"]
		&&
		{T^3}Y
		\arrow[dddd,"{\mu}_{TY}"]
		\\
		&&&&
		{}
		\arrow[dd,Rightarrow,shorten=10,"{T^2}{\lambda}_Y"]
		\\
		&
		{}
		\arrow[dd,Rightarrow,red,"{\mu}_{{\eta}_X}"]
		&&&&
		{}
		\arrow[dd,Rightarrow,"{\mu}_{{\mu}_Y}"]
		\\
		&&
		\color{blue}{T^3}Y
		\arrow[rr,blue,"{T^3}g"]
		\arrow[ddd,blue,"{\mu}_{TX}"]
		&
		{}
		\arrow[ddd,Rightarrow,shorten=10,red,"{\mu}_{Tg}"]
		&
		{T^4}Y
		\arrow[rruuu,"{T^2}{\mu}_Y" description]
		\arrow[ddd,"{\mu}_{{T^2}Y}"]
		\\
		TX
		\arrow[rrdd,"T{\eta}_X"']
		&
		{}
		&&&&
		{}
		&
		{T^2}Y
		\\
		\\
		&&
		{T^2}X
		\arrow[rr,"{T^2}g"']
		&
		{}
		&
		{T^3}Y
		\arrow[rruu,"T{\mu}_Y"']
	\end{tikzcd}
\end{proof}

\subsection{Proof of Proposition \ref{cones to theta pseudonat functor}}\label{proof Proposition cones to theta pseudonat functor}

\begin{proof}
	\noindent We need to verify that $\underline{J}\left(g, \bar{g}\right)$ is well-defined as a morphism equipped with a thunking. For the unit condition, begin with the pasting below.
	
	$$
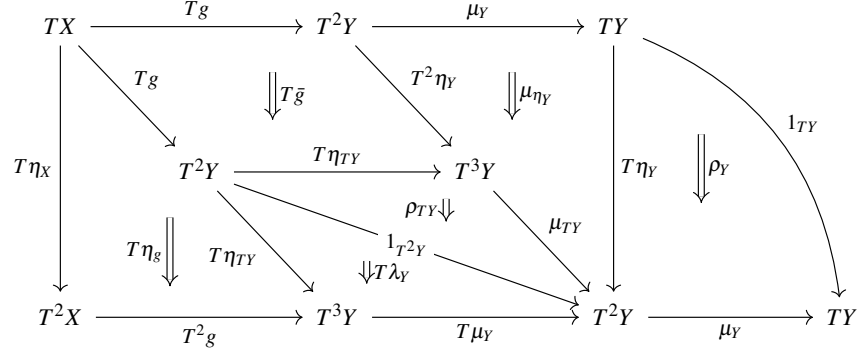
\begin{tikzcd}[font=\fontsize{9}{6}, row sep = 15]
		TX
		\arrow[dddd,"T{\eta}_X"']
		\arrow[rr,"Tg"]
		\arrow[rdd,"Tg"]
		&
		{} \arrow[dd,Rightarrow,shorten=10,shift left=10,"T\bar{g}"]
		&
		{T^2}Y
		\arrow[rdd,"{T^2}{\eta}_Y"]
		\arrow[rr,"{\mu}_Y"]
		&
		{} \arrow[dd,Rightarrow,shorten=10,shift left=5,"{\mu}_{{\eta}_Y}"]
		&
		TY
		\arrow[dddd,"T{\eta}_Y"]
		\arrow[rrdddd,bend left,"1_{TY}"]
		\\
		&&&&& {} \arrow[dd,Rightarrow,shorten=10,shift right=4,"{\rho}_Y"]
		\\
		&
		{T^2}Y
		\arrow[rr,"T{\eta}_{TY}"]
		\arrow[rrrdd,"1_{{T^2}Y}" description]
		\arrow[rdd,"T{\eta}_{TY}"']
		\arrow[dd,Rightarrow,shorten=10,shift right=4,"T{\eta}_g"']
		&&
		{T^3}Y
		\arrow[rdd,"{\mu}_{TY}"]
		\arrow[d,Rightarrow,shift right=4,shorten=3,"{\rho}_{TY}"']
		\\
		&&
		{} \arrow[d,Rightarrow,shorten=3,shift left=4,"T{\lambda}_Y"]
		&
		{}
		&& {}
		\\
		{T^2}X
		\arrow[rr,"{T^2}g"']
		&
		{}
		&
		{T^3}Y
		\arrow[rr,"T{\mu}_Y"']
		&&
		{T^2}Y
		\arrow[rr,"{\mu}_Y"']
		&&
		TY
	\end{tikzcd}$$
	\noindent Apply coherence $5$ for $\left(\mathcal{B}, T\right)$ to reduce the pasting above to the pasting below.
	
	$$
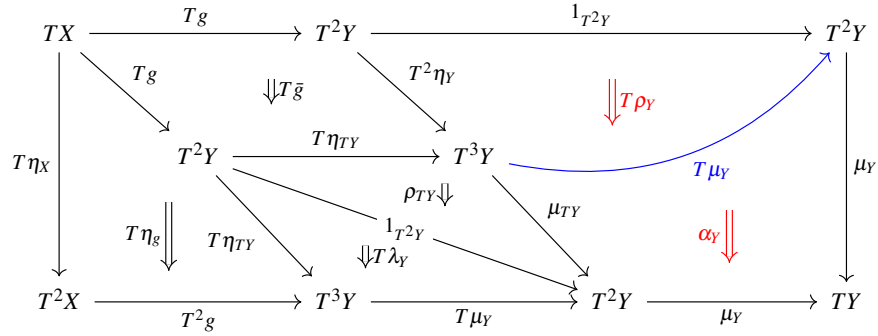
\begin{tikzcd}[font=\fontsize{9}{6}, row sep = 15]
		TX
		\arrow[dddd,"T{\eta}_X"']
		\arrow[rr,"Tg"]
		\arrow[rdd,"Tg"]
		&
		{} \arrow[dd,Rightarrow,shorten=10,shift left=10,"T\bar{g}"]
		&
		{T^2}Y
		\arrow[rdd,"{T^2}{\eta}_Y"]
		\arrow[rrrr,"1_{{T^2}Y}"]
		&&
		{} \arrow[dd,Rightarrow,shorten=10,"T{\rho}_Y",red]
		&&
		{T^2}Y
		\arrow[dddd,"{\mu}_Y"]
		\\
		\\
		&
		{T^2}Y
		\arrow[rr,"T{\eta}_{TY}"]
		\arrow[rrrdd,"1_{{T^2}Y}" description]
		\arrow[rdd,"T{\eta}_{TY}"']
		\arrow[dd,Rightarrow,shorten=10,shift right=4,"T{\eta}_g"']
		&&
		{T^3}Y
		\arrow[rdd,"{\mu}_{TY}"]
		\arrow[d,Rightarrow,shift right=4,shorten=3,"{\rho}_{TY}"']
		\arrow[rrruu,"T{\mu}_Y"',blue,bend right]
		&
		{}
		&
		{}
		\arrow[dd,Rightarrow,shorten=13,red,"{\alpha}_Y"']
		\\
		&&
		{} \arrow[d,Rightarrow,shorten=3,shift left=4,"T{\lambda}_Y"]
		&
		{}
		&& {}
		\\
		{T^2}X
		\arrow[rr,"{T^2}g"']
		&
		{}
		&
		{T^3}Y
		\arrow[rr,"T{\mu}_Y"']
		&&
		{T^2}Y
		\arrow[rr,"{\mu}_Y"']
		&
		{}
		&
		TY
	\end{tikzcd}$$
	
	\noindent Apply the unit coherence for $\left(g, \bar{g}\right)$ to reduce the pasting above to the pasting below.
	
	$$
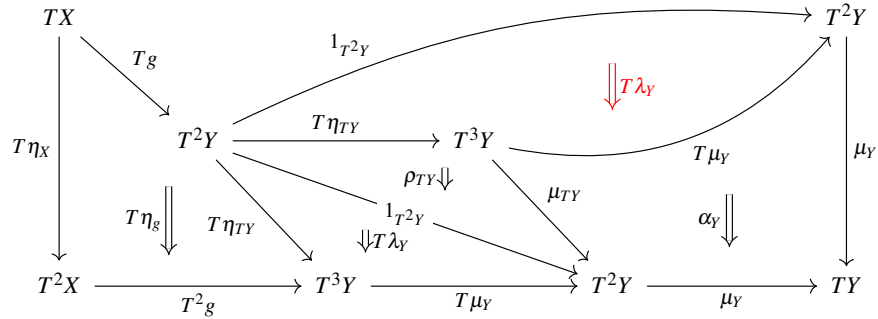
\begin{tikzcd}[font=\fontsize{9}{6}, row sep = 15]
		TX
		\arrow[dddd,"T{\eta}_X"']
		\arrow[rdd,"Tg"]
		&&&&
		{} \arrow[dd,Rightarrow,shorten=10,"T{\lambda}_Y",red]
		&&
		{T^2}Y
		\arrow[dddd,"{\mu}_Y"]
		\\
		\\
		&
		{T^2}Y
		\arrow[rr,"T{\eta}_{TY}"]
		\arrow[rrrdd,"1_{{T^2}Y}" description]
		\arrow[rdd,"T{\eta}_{TY}"']
		\arrow[dd,Rightarrow,shorten=10,shift right=4,"T{\eta}_g"']
		\arrow[rrrrruu,bend left=15,"1_{{T^2}Y}" near start]
		&&
		{T^3}Y
		\arrow[rdd,"{\mu}_{TY}"]
		\arrow[d,Rightarrow,shift right=4,shorten=3,"{\rho}_{TY}"']
		\arrow[rrruu,"T{\mu}_Y"',bend right]
		&
		{}
		&
		{}
		\arrow[dd,Rightarrow,shorten=13,"{\alpha}_Y"']
		\\
		&&
		{} \arrow[d,Rightarrow,shorten=3,shift left=4,"T{\lambda}_Y"]
		&
		{}
		&& {}
		\\
		{T^2}X
		\arrow[rr,"{T^2}g"']
		&
		{}
		&
		{T^3}Y
		\arrow[rr,"T{\mu}_Y"']
		&&
		{T^2}Y
		\arrow[rr,"{\mu}_Y"']
		&
		{}
		&
		TY
	\end{tikzcd}$$
	
	\noindent Apply coherence $2$ for the pseudomonad $\left(\mathcal{B}, T\right)$ to reduce the pasting above to the pasting below.
	
	$$\begin{tikzcd}[font=\fontsize{9}{6}, row sep = 15]
		TX
		\arrow[dd,"T{\eta}_X"']
		\arrow[r,"Tg"]
		&
		{T^2}Y
		\arrow[rrrdd,"1_{{T^2}Y}" description]
		\arrow[rdd,"T{\eta}_{TY}"']
		\arrow[dd,Rightarrow,shorten=10,shift right=4,"T{\eta}_g"']
		\\
		&&
		{} \arrow[d,Rightarrow,shorten=3,shift left=4,"T{\lambda}_Y"]
		&
		{}
		&& {}
		\\
		{T^2}X
		\arrow[rr,"{T^2}g"']
		&
		{}
		&
		{T^3}Y
		\arrow[rr,"T{\mu}_Y"']
		&&
		{T^2}Y
		\arrow[rr,"{\mu}_Y"']
		&
		{}
		&
		TY
	\end{tikzcd}$$
	
	\noindent Apply coherence $2$ once again to reduce the pasting above to the pasting below.
	
	$$\begin{tikzcd}[font=\fontsize{9}{6}, row sep = 15]
		TX
		\arrow[dd,"T{\eta}_X"']
		\arrow[r,"Tg"]
		&
		{T^2}Y
		\arrow[rdd,"T{\eta}_{TY}"']
		\arrow[dd,Rightarrow,shorten=10,shift right=4,"T{\eta}_g"']
		\arrow[rrr,"1_{{T^2}Y}"]
		&
		{} \arrow[dd,Rightarrow,shorten=10,red,"{\rho}_{TY}"]
		&&
		{T^2}Y
		\arrow[rrdd,"{\mu}_Y"]
		\arrow[dd,Rightarrow,shorten=10,red,"{\alpha}_Y"]
		\\
		&&&
		{}
		&& {}
		\\
		{T^2}X
		\arrow[rr,"{T^2}g"']
		&
		{}
		&
		{T^3}Y
		\arrow[rr,"T{\mu}_Y"']
		\arrow[rruu,blue,"{\mu}_{TY}" description]
		&&
		{T^2}Y
		\arrow[rr,"{\mu}_Y"']
		&
		{}
		&
		TY
	\end{tikzcd}$$
	
	\noindent Finally, apply the modification coherence for $\rho$ on $g$ to reduce the pasting above to the pasting below.
	
	$$\begin{tikzcd}[font=\fontsize{9}{6}, row sep = 15]
		TX
		\arrow[rr,"T{\eta}_X"]
		\arrow[rrdd,"1_{TX}"']
		&
		{}
		\arrow[d,Rightarrow,shorten=3,shift left=3,red,"{\rho}_X"]
		&
		{T^2}X
		\arrow[rr,"{T^2}g"]
		\arrow[dd,"{\mu}_X",blue]
		&
		{} \arrow[dd,shorten=15,Rightarrow,red,"{\mu}_g"]
		&
		{T^3}Y
		\arrow[rr,"T{\mu}_Y"]
		\arrow[dd,"{\mu}_{TY}"]
		&
		{}
		\arrow[dd,shorten=15,Rightarrow,"{\alpha}_Y"]
		&
		{T^2}Y
		\arrow[dd,"{\mu}_Y"]
		\\
		& {}
		\\
		&&
		TX
		\arrow[rr,"Tg"']
		&
		{}
		&
		{T^2}Y
		\arrow[rr,"{\mu}_Y"']
		&
		{}
		&
		TY
	\end{tikzcd}$$
	
	\noindent We refer the reader to Appendix 11.2 of \cite{Miranda PhD} for the proof of the associativity condition, which is similar but longer than the proof of the unit condition. Given $\phi: \left(g, \bar{g}\right) \rightarrow \left(h, \bar{h}\right)$ in $\mathbf{Cone}_{T}\left(X, Y\right)$, the fact that $\underline{J}\left(\phi\right)$ is thunkable follows from the condition on morphisms in $\mathbf{Cone}_{T}\left(X, Y\right)$ and pseudonaturality of $T\eta$ on $\phi$. Functoriality of $\underline{J}$ is clear so the proof of Proposition \ref{cones to theta pseudonat functor} is complete.
\end{proof}

\subsection{Proof of Lemma \ref{cones to free theta pseudonat essentially surjective}}\label{proof Lemma cones to free theta pseudonat essentially surjective}

\begin{proof}
	\noindent We give details only for the cocycle condition, referring the reader to Appendix 11.3 of \cite{Miranda PhD} for the proof of the unit condition. Begin with the pasting depicted below left. Apply pseudonaturality of $\eta$ on $\theta_{p}$ to arrive at the pasting below right.
	
	$$\begin{tikzcd}[column sep = 13, ,font=\fontsize{9}{6}, row sep = 15]
		&&&
		TY
		\arrow[rrdd,"T{\eta}_Y", bend left=20]
		\arrow[dd,shorten=8,"{\theta}_p",Rightarrow]
		\\
		\\
		&
		TX
		\arrow[rruu,bend left=20,"p"]
		\arrow[rr,"T{\eta}_X" description]
		&
		{}
		\arrow[dd,Rightarrow,shorten=15,shift right=10,"{\eta}_{{\eta}_X}"]
		&
		{T^2}X
		\arrow[rr,"Tp" description]
		&
		{}
		\arrow[dd,Rightarrow,shorten=15,shift right=10,"{\eta}_{p}"]
		&
		{T^2}Y
		\arrow[rdd,"{T^2}{\eta}_Y"]
		\arrow[dddd,shorten=30,Rightarrow,"{\eta}_{T{\eta}_Y}"]
		\\
		&&&&
		{}
		\\
		X
		\arrow[ruu,"{\eta}_X"]
		\arrow[rdd,"{\eta}_X"']
		\arrow[rr,"{\eta}_X" description]
		&&
		TX
		\arrow[ruu,"{\eta}_{TX}" description]
		\arrow[rdd,"T{\eta}_X" description]
		\arrow[dd,Rightarrow,shorten=15,shift right=10,"{\eta}_{{\eta}_X}"]
		\arrow[rr,"p" description]
		&
		{}
		&
		TY
		\arrow[ruu,"{\eta}_{TY}" description]
		\arrow[rdd,"T{\eta}_Y" description]
		\arrow[dd,Rightarrow,shorten=15,shift right=10,"{\theta}_{p}"]
		&&
		{T^3}Y
		\\
		\\
		&
		TX
		\arrow[rr,"{\eta}_{TX}" description]
		\arrow[rrdd,"p"',bend right=20]
		&
		{}
		&
		{T^2}X
		\arrow[rr,"Tp" description]
		\arrow[dd,shorten=8,Rightarrow,"{\eta}_p"]
		&
		{}
		&
		{T^2}Y
		\arrow[ruu,"{\eta}_{{T^2}Y}"']
		\\
		\\
		&&&
		TY
		\arrow[rruu,"{\eta}_{TY}"',bend right=20]
	\end{tikzcd}\begin{tikzcd}[column sep = 13, ,font=\fontsize{9}{6}, row sep = 15]
		&&&
		TY
		\arrow[rrdd,"T{\eta}_Y", bend left=20]
		\arrow[dd,shorten=8,"{\theta}_p",Rightarrow]
		\\
		\\
		&
		TX
		\arrow[rruu,bend left=20,"p"]
		\arrow[rr,"T{\eta}_X" description]
		&
		{}
		\arrow[dd,Rightarrow,shorten=15,shift right=10,"{\eta}_{{\eta}_X}"]
		&
		{T^2}X
		\arrow[rr,"Tp" description]
		\arrow[rdd,"{T^2}{\eta}_X" description,blue]
		\arrow[dddd,shorten=30,Rightarrow,red,"{\eta}_{{T\eta}_X}"]
		&
		{}
		\arrow[dd,Rightarrow,shorten=15,shift left=10,"T{\theta}_{p}",red]
		&
		{T^2}Y
		\arrow[rdd,"{T^2}{\eta}_Y"]
		\\
		&&&&
		{}
		\\
		X
		\arrow[ruu,"{\eta}_X"]
		\arrow[rdd,"{\eta}_X"']
		\arrow[rr,"{\eta}_X" description]
		&&
		TX
		\arrow[ruu,"{\eta}_{TX}" description]
		\arrow[rdd,"T{\eta}_X" description]
		\arrow[dd,Rightarrow,shorten=15,shift right=10,"{\eta}_{{\eta}_X}"]
		&
		{}
		&
		\color{blue}{T^3}X
		\arrow[rr,"{T^2}p" description,blue]
		\arrow[dd,Rightarrow,shorten=15,shift left=10,"{\eta}_{Tp}",red]
		&&
		{T^3}Y
		\\
		\\
		&
		TX
		\arrow[rr,"{\eta}_{TX}" description]
		\arrow[rrdd,"p"',bend right=20]
		&
		{}
		&
		{T^2}X
		\arrow[rr,"Tp" description]
		\arrow[ruu,"{\eta}_{{T^2}X}" description,blue]
		\arrow[dd,shorten=8,Rightarrow,"{\eta}_p"]
		&
		{}
		&
		{T^2}Y
		\arrow[ruu,"{\eta}_{{T^2}Y}"']
		\\
		\\
		&&&
		TY
		\arrow[rruu,"{\eta}_{TY}"',bend right=20]
	\end{tikzcd}$$

	\noindent Apply pseudonaturality of $\eta$ on $\eta_{X}$ to reduce the pasting above right to the pasting below left. Apply the associativity coherence for $\left(\left(p, \bar{p}\right), \theta_{p}\right)$ as a  $1$-cell in $\left({\B}_{T}\right)_{\theta}$ to reduce the pasting below left to the pasting below right.
	
	$$
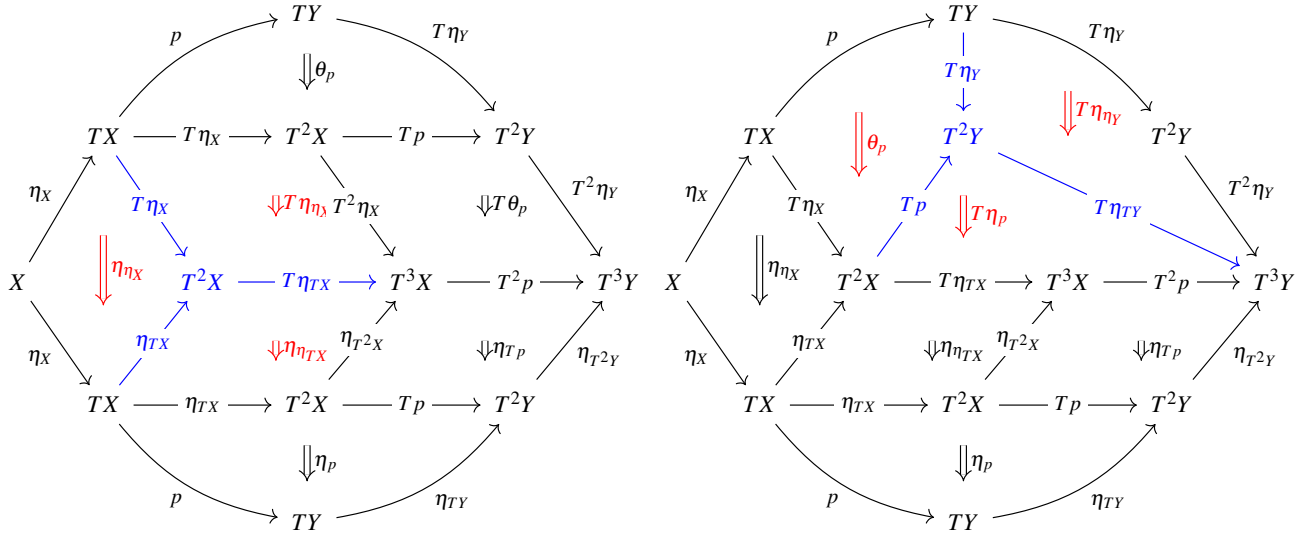
\begin{tikzcd}[column sep = 13, ,font=\fontsize{9}{6}, row sep = 15]
		&&&
		TY
		\arrow[rrdd,"T{\eta}_Y", bend left=20]
		\arrow[dd,shorten=8,"{\theta}_p",Rightarrow]
		\\
		\\
		&
		TX
		\arrow[rruu,bend left=20,"p"]
		\arrow[rdd,"T{\eta}_X" description,blue]
		\arrow[rr,"T{\eta}_X" description]
		\arrow[dddd,shorten=30,Rightarrow,"{\eta}_{{\eta}_X}",red]
		&
		{}
		\arrow[dd,Rightarrow,shorten=15,shift left=10,"T{\eta}_{{\eta}_{X}}",red]
		&
		{T^2}X
		\arrow[rr,"Tp" description]
		\arrow[rdd,"{T^2}{\eta}_X" description]
		&
		{}
		\arrow[dd,Rightarrow,shorten=15,shift left=10,"T{\theta}_{p}"]
		&
		{T^2}Y
		\arrow[rdd,"{T^2}{\eta}_Y"]
		\\
		&&&&
		{}
		\\
		X
		\arrow[ruu,"{\eta}_X"]
		\arrow[rdd,"{\eta}_X"']
		&&
		\color{blue}{T^2}X
		\arrow[rr,"T{\eta}_{TX}" description,blue]
		\arrow[dd,Rightarrow,shorten=15,shift left=10,"{\eta}_{{\eta}_{TX}}",red]
		&
		{}
		&
		{T^3}X
		\arrow[rr,"{T^2}p" description]
		\arrow[dd,Rightarrow,shorten=15,shift left=10,"{\eta}_{Tp}"]
		&&
		{T^3}Y
		\\
		\\
		&
		TX
		\arrow[ruu,"{\eta}_{TX}" description,blue]
		\arrow[rr,"{\eta}_{TX}" description]
		\arrow[rrdd,"p"',bend right=20]
		&
		{}
		&
		{T^2}X
		\arrow[rr,"Tp" description]
		\arrow[ruu,"{\eta}_{{T^2}X}" description]
		\arrow[dd,shorten=8,Rightarrow,"{\eta}_p"]
		&
		{}
		&
		{T^2}Y
		\arrow[ruu,"{\eta}_{{T^2}Y}"']
		\\
		\\
		&&&
		TY
		\arrow[rruu,"{\eta}_{TY}"',bend right=20]
	\end{tikzcd}\begin{tikzcd}[column sep = 13, ,font=\fontsize{9}{6}, row sep = 15]
		&&
		{}
		\arrow[dddd,shorten=30,Rightarrow,red,"{\theta}_p"]
		&
		TY
		\arrow[dd,"T{\eta}_Y" description, blue]
		\arrow[rrdd,"T{\eta}_Y", bend left=20]
		&
		{}
		\arrow[ddd,shorten=22,Rightarrow,red,"T{\eta}_{{\eta}_Y}"]
		\\
		\\
		&
		TX
		\arrow[rruu,bend left=20,"p"]
		\arrow[rdd,"T{\eta}_X" description]
		\arrow[dddd,shorten=30,Rightarrow,"{\eta}_{{\eta}_X}"]
		&&
		\color{blue}{T^2}Y
		\arrow[rrrdd,"T{\eta}_{TY}" description,blue]
		\arrow[dd,shorten=15,Rightarrow,red,"T{\eta}_p"]
		&&
		{T^2}Y
		\arrow[rdd,"{T^2}{\eta}_Y"]
		\\
		&&&&
		{}
		\\
		X
		\arrow[ruu,"{\eta}_X"]
		\arrow[rdd,"{\eta}_X"']
		&&
		{T^2}X
		\arrow[rr,"T{\eta}_{TX}" description]
		\arrow[ruu,blue,"Tp" description]
		\arrow[dd,Rightarrow,shorten=15,shift left=10,"{\eta}_{{\eta}_{TX}}"]
		&
		{}
		&
		{T^3}X
		\arrow[rr,"{T^2}p" description]
		\arrow[dd,Rightarrow,shorten=15,shift left=10,"{\eta}_{Tp}"]
		&&
		{T^3}Y
		\\
		\\
		&
		TX
		\arrow[ruu,"{\eta}_{TX}" description]
		\arrow[rr,"{\eta}_{TX}" description]
		\arrow[rrdd,"p"',bend right=20]
		&
		{}
		&
		{T^2}X
		\arrow[rr,"Tp" description]
		\arrow[ruu,"{\eta}_{{T^2}X}" description]
		\arrow[dd,shorten=8,Rightarrow,"{\eta}_p"]
		&
		{}
		&
		{T^2}Y
		\arrow[ruu,"{\eta}_{{T^2}Y}"']
		\\
		\\
		&&&
		TY
		\arrow[rruu,"{\eta}_{TY}"',bend right=20]
	\end{tikzcd}$$
	
	\noindent Finally, apply pseudonaturality of $\eta$ on $p$ to reduce the pasting above right to the pasting below, and observe that this completes the proof.
	
	$$
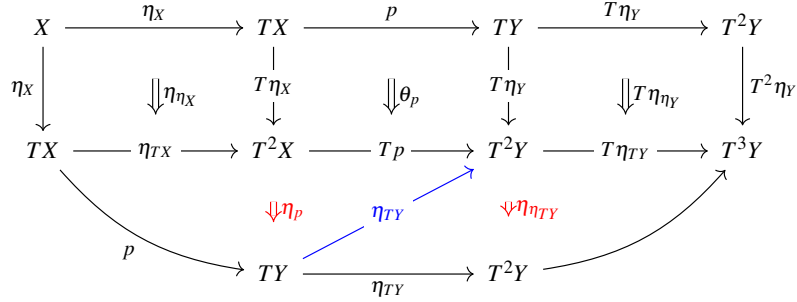
\begin{tikzcd}[font=\fontsize{9}{6}, row sep = 15]
		X
		\arrow[rr,"{\eta}_X"]
		\arrow[dd,"{\eta}_X"']
		&
		{}
		\arrow[dd,shorten=12,Rightarrow,"{\eta}_{{\eta}_X}"]
		&
		TX
		\arrow[rr,"p"]
		\arrow[dd,"T{\eta}_X" description]
		&
		{}
		\arrow[dd,shorten=12,Rightarrow,"{\theta}_{p}"]
		&
		TY
		\arrow[rr,"T{\eta}_Y"]
		\arrow[dd,"T{\eta}_Y" description]
		&
		{}
		\arrow[dd,shorten=12,Rightarrow,"T{\eta}_{{\eta}_Y}"]
		&
		{T^2}Y
		\arrow[dd,"{T^2}{\eta}_{Y}"]
		\\
		\\
		TX
		\arrow[rr,"{\eta}_{TX}" description]
		\arrow[rrdd,bend right=20,"p"']
		&
		{}
		&
		{T^2}X
		\arrow[rr,"Tp" description]
		\arrow[dd,shorten=12,Rightarrow,"{\eta}_{p}",red]
		&
		{}
		&
		{T^2}Y
		\arrow[rr,"T{\eta}_{TY}" description]
		\arrow[dd,shorten=12,Rightarrow,"{\eta}_{{\eta}_{TY}}",red]
		&
		{}
		&
		{T^3}Y
		\\
		\\
		&&
		TY
		\arrow[rr,"{\eta}_{TY}"']
		\arrow[rruu,"{\eta}_{TY}" description,blue]
		&&
		{T^2}Y
		\arrow[rruu,bend right=20]
	\end{tikzcd}$$
\end{proof}

\subsection{Proof of Proposition \ref{cones to free theta pseudonat an equivalence}}\label{proof cones to free theta pseudonat an equivalence}

For essential surjectivity on objects, we claim that the image under $\underline{J}_{X, Y}$ of the cone from $X$ to $Y$ defined in Lemma \ref{cones to free theta pseudonat essentially surjective} is isomorphic to $\left(\left(p, \bar{p}\right), \theta_{p}\right)$, via the $2$-cell of pseudoalgebras depicted below.

$$
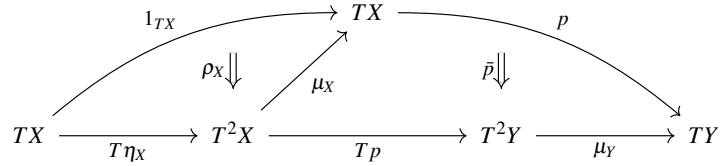
\begin{tikzcd}[row sep = 15, font=\fontsize{9}{6}]
	&&
	{}
	\arrow[dd,Rightarrow,shorten=8,"{\rho}_X"']
	&
	TX
	\arrow[rrrdd,bend left=20,"p"]
	&
	{}
	\arrow[dd,Rightarrow,shorten=8,"\bar{p}"']
	\\
	\\
	TX
	\arrow[rrruu,bend left=20,"1_{TX}"]
	\arrow[rr,"T{\eta}_X"']
	&&
	{T^2}X
	\arrow[ruu,"{\mu}_X"']
	\arrow[rr,"Tp"']
	&
	{}
	&
	{T^2}Y
	\arrow[rr,"{\mu}_Y"']
	&&
	TY
\end{tikzcd}$$

\noindent Thunkability of this $2$-cell follows from a diagram chase given in Appendix 11.4 of \cite{Miranda PhD}. For fully faithfulness, we already know that arbitrary $2$-cells $\phi$ from $g: X \rightarrow TY$ to $h: X \rightarrow TY$ are in bijection with $2$-cells of free pseudoalgebras $\left(\mu_{Y}, \alpha_{Y}\right).\left(Tg, \mu_{g}\right) \Rightarrow \left(\mu_{Y}, \alpha_{Y}\right).\left(Th, \mu_{h}\right)$. This is part of the biequivalence between the Kleisli bicategory of $\left(A, T\right)$ and ${\B}_{T}$. It therefore suffices to observe that if $\phi: \left(\left(p,\bar{p}\right), \theta_{p}\right) \Rightarrow \left(\left(q, \bar{q}\right), \theta_{q}\right)$ is thunkable then $\psi.\eta_{X}$ is a morphism in $\mathbf{Cone}_{T}\left(X, Y\right)$. This follows from pseudonaturality of $\eta$ on $\phi$.

\subsection{Proof of Proposition \ref{G isomorphic to restriction along eta}}\label{Proof of Proposition G isomorphic to restriction along eta}

For part (1), begin with the pasting depicted below top left and apply the modification coherence for $\rho$ on $\eta_{Y}$ to reduce to the pasting depicted below top right. Finally, apply coherence $3$ for $\left(\mathcal{B} ,T\right)$ to arrive at the pasting depicted below and observe that this completes the proof. For part $(2)$, naturality follows by middle-four interchange.

$$\begin{tikzcd}[column sep = 15, ,font=\fontsize{9}{6}]
	TX
	\arrow[rr,"Tg"]
	\arrow[dddd,"T{\eta}_X"']
	&&
	TY
	\arrow[rr,"T{\eta}_Y" description]
	\arrow[rrrr,bend left,"1_{TY}" name=A]
	\arrow[dd,"T{\eta}_Y"]
	\arrow[dddd,bend right,"T{\eta}_Y" description]
	&
	{}
	\arrow[dd,Rightarrow,shorten=15,"T{\eta}_{{\eta}_Y}"]
	&
	{T^2}Y
	\arrow[from=A,Rightarrow,"{\rho}_Y",shorten=3]
	\arrow[rr,"{\mu}_X"']
	\arrow[dd,"{T^2}{\eta}_Y"]
	&
	{} \arrow[dd,Rightarrow,shorten=15,"{\mu}_{{\eta}_Y}"]
	&
	TY
	\arrow[dddd,"T{\eta}_Y"]
	\\
	&
	{}
	\arrow[dd,Rightarrow,shorten=10,"T{\eta}_g"',shift right=5]
	\\
	&&
	{T^2}Y
	\arrow[rr,"T{\eta}_{TY}"]
	\arrow[rrdd,"T{\eta}_{TY}" description,near end]
	\arrow[rrrrdd,"1_{{T^2}Y}" description,near start]
	\arrow[dd,Rightarrow,shorten=10,"T{\eta}_{{\eta}_Y}"]
	&
	{}
	&
	{T^3}Y
	\arrow[rrdd,"{\mu}_{TY}"]
	\arrow[d,Rightarrow,shorten=3,"{\rho}_{TY}"']
	&
	{}
	\\
	&
	{}
	&&&
	{}
	\arrow[d,Rightarrow,shorten=3,"T{\lambda}_{Y}"]
	\\
	{T^2}X
	\arrow[rr,"{T^2}g"']
	&&
	{T^2}Y
	\arrow[rr,"{T^2}{\eta}_Y"']
	&&
	{T^3}Y
	\arrow[rr,"T{\mu}_Y"']
	&&
	{T^2}Y&{}
\end{tikzcd}\begin{tikzcd}[font=\fontsize{9}{6}]
	TX
	\arrow[rr,"Tg"]
	\arrow[dddd,"T{\eta}_X"']
	&&
	TY
	\arrow[dd,"T{\eta}_Y"]
	\arrow[ddddl,bend right,"T{\eta}_Y"']
	\\
	&
	{}
	\arrow[dd,Rightarrow,shorten=10,"T{\eta}_g"', shift right = 10]
	\\
	&&
	{T^2}Y
	\arrow[dd,"T{\eta}_{TY}" description,near end]
	\arrow[rdd, bend left = 30, "1_{{T^2}Y}",near start]
	\arrow[ddl,Rightarrow,shorten=10,"T{\eta}_{{\eta}_Y}"']
	&
	{}
	&&
	{}
	\\
	&
	{}
	&
	{}
	\arrow[d,Rightarrow,shorten=3,"T{\lambda}_{Y}", shift left = 8]
	\\
	{T^2}X
	\arrow[r,"{T^2}g"']
	&
	{T^2}Y
	\arrow[r,"{T^2}{\eta}_Y"']
	&
	{T^3}Y
	\arrow[r,"T{\mu}_Y"']
	&
	{T^2}Y
\end{tikzcd}$$

$$\begin{tikzcd}[font=\fontsize{9}{6}]
	TX
	\arrow[rr,"Tg"]
	\arrow[dd,"T{\eta}_X"']
	&
	{}
	\arrow[dd,Rightarrow,shorten=10,"T{\eta}_g"]
	&
	TY
	\arrow[dd,"T{\eta}_Y"]
	\\
	\\
	{T^2}X
	\arrow[rr,"{T^2}g"']
	&
	{}
	&
	{T^2}Y
	\arrow[rr,"{T^2}{\eta}_Y"]
	\arrow[rrrr,bend right,"1_{{T^2}Y}"' name=A]
	&&
	{T^3}Y
	\arrow[rr,"T{\mu}_Y"]
	\arrow[to=A,Rightarrow,shorten=3,"T{\rho}_Y",red]
	&&
	{T^2}Y
\end{tikzcd}$$

\nocite{*}
\end{document}